\documentclass[final,intlimits]{article}

\usepackage{mathtools}
\usepackage{amssymb}
\usepackage{amsthm}
\usepackage{xcolor}
\usepackage{tikz}
\usepackage{fullpage}

\DeclareSymbolFont{bbold}{U}{bbold}{m}{n}
\DeclareSymbolFontAlphabet{\mathbbold}{bbold}

\newcommand{\N}{\mathbb{N}}
\newcommand{\Z}{\mathbb{Z}}

\newcommand{\R}{\mathbb{R}}
\newcommand{\C}{\mathbb{C}}

\newcommand{\1}{\mathbbold{1}}
\newcommand{\calL}{\mathcal{L}}
\newcommand{\calK}{\mathcal{K}}
\newcommand{\calT}{\mathcal{T}}

\newcommand{\BO}{\mathrm{BO}}
\newcommand{\BDO}{\mathrm{BDO}}

\DeclareMathOperator{\ran}{ran}

\newcommand{\ess}{{\rm ess}}

\newcommand{\calA}{\mathcal{A}}
\DeclareMathOperator*{\slim}{s-lim\,}
\DeclareMathOperator*{\wlim}{w-lim\,}
\newcommand{\bdry}[1]{\Gamma #1}
\newcommand{\A}{A}
\newcommand{\Aprime}{A$^{\prime}$}

\newcommand{\dist}{\mathrm{dist}}



\DeclareMathOperator{\prop}{prop}
\DeclareMathOperator{\spt}{spt}
\DeclareMathOperator{\diam}{diam}

\renewcommand{\Re}{\operatorname{Re}}
\renewcommand{\Im}{\operatorname{Im}}

\providecommand{\dupa}[2]{\left\langle #1 , #2 \right\rangle}

\newcommand{\from}{\colon}

\let\leq\leqslant

\let\geq\geqslant

\makeatletter
\def\@row#1,{#1\@ifnextchar;{\@gobble}{&\@row}}
\def\@matrix{%
    \expandafter\@row\my@arg,;%
    \@ifnextchar({\\ \get@in@paren{\@matrix}}{\after@matrix}%
    }
\def\matrixtype#1#2#3{%
    \ifmmode\def\after@matrix{\end{#2}\right#3}%
    \else\def\after@matrix{\end{#2}\right#3$}$\fi\iffalse\fi
    \left#1\begin{#2}\get@in@paren{\@matrix}%
    }
\def\@column#1,{#1\@ifnextchar;{\@gobble}{\\ \@column}}
\newcommand\vect{}
\def\svect(#1){\left(\begin{smallmatrix}\@column#1,;\end{smallmatrix}\right)}
\def\vect{\get@in@paren{\@vect}}
\def\@vect{\left(\begin{matrix}\expandafter\@column\my@arg,;\end{matrix}\right)}
\def\get@in@paren#1({\def\my@arg{}\def\my@rest{}\def\after@get{#1}\get@arg}
\let\e@a\expandafter
\def\get@arg#1){\e@a\kl@test\my@rest#1(;}
\def\kl@test#1(#2;{\e@a\def\e@a\my@arg\e@a{\my@arg#1}%
                   \ifx:#2:\let\my@exec\after@get
                   \else\let\my@exec\get@arg
                        \e@a\def\e@a\my@arg\e@a{\my@arg(}%
                        \def@rest#2;%
                   \fi\my@exec}
\def\def@rest#1(;{\def\my@rest{#1\kl@zu}}
\def\kl@zu{)}

\makeatletter
\newcommand\MyPairedDelimiter{%
  \@ifstar{\My@Paired@Delimiter{{}}}
          {\My@Paired@Delimiter{}}%
}
\newcommand\My@Paired@Delimiter[4]{%
  \newcommand#2{%
    \@ifstar{\start@PD{#1}{\delimitershortfall=-1sp}{#3}{#4}}
            {\start@PD{#1}{}{#3}{#4}}%
  }%
}
\newcommand\start@PD[5]{%
  #1\mathopen{\mathpalette\put@delim@helper{\put@delim{#2}{#3}{.}{#5}}}%
  #5%
  \mathclose{\mathpalette\put@delim@helper{\put@delim{#2}{.}{#4}{#5}}}%
}
\newcommand\put@delim@helper[2]{%
  \hbox{$\m@th\nulldelimiterspace=0pt #2#1$}%
}
\newcommand\put@delim[5]{%
  \setbox\z@\hbox{$\m@th#5{#4}$}%
  \setbox\tw@\null
  \ht\tw@\ht\z@ \dp\tw@\dp\z@
  #1#5%
  \left#2\box\tw@\right#3%
}

\makeatother
\MyPairedDelimiter*{\abs}{\lvert}{\rvert}
\MyPairedDelimiter*{\norm}{\lVert}{\rVert}
\MyPairedDelimiter{\set}{\{}{\}}
\newcommand{\tripnorm}[1]{{\left\vert\kern-0.25ex\left\vert\kern-0.25ex\left\vert #1 
    \right\vert\kern-0.25ex\right\vert\kern-0.25ex\right\vert}}

\theoremstyle{plain} % default
\newtheorem{theorem}{Theorem}[section]
\newtheorem{corollary}[theorem]{Corollary}
\newtheorem{lemma}[theorem]{Lemma}
\newtheorem{proposition}[theorem]{Proposition}

\newtheorem{assumption}[theorem]{Assumption}
\theoremstyle{definition}

\newtheorem{definition}[theorem]{Definition}
\newtheorem{remark}[theorem]{Remark}

\usepackage{enumitem}

\setenumerate[1]{nolistsep} % Only the level 1
\setenumerate[2]{nolistsep} % Only the level 2

\newcommand{\Hmm}[1]{\leavevmode{\marginpar{\tiny%
$\hbox to 0mm{\hspace*{-0.5mm}$\leftarrow$\hss}%
\vcenter{\vrule depth 0.1mm height 0.1mm width \the\marginparwidth}%
\hbox to 0mm{\hss$\rightarrow$\hspace*{-0.5mm}}$\\\relax\raggedright #1}}}

\allowdisplaybreaks

\begin{document}

\medmuskip=4mu plus 2mu minus 3mu
\thickmuskip=5mu plus 3mu minus 1mu
\belowdisplayshortskip=9pt plus 3pt minus 5pt

\title{Limit operators techniques on general metric measure spaces of bounded geometry}

\author{Raffael Hagger\footnote{R.H. has received funding from the European Union's Horizon 2020 research and innovation
programme under the Marie Sklodowska-Curie grant agreement No 844451.}\; and Christian Seifert}

\date{\today}

\maketitle

\begin{abstract}
  We study band-dominated operators on (subspaces of) $L_p$-spaces over metric measure spaces of bounded geometry satisfying an additional property. We single out core assumptions to obtain, in an abstract setting, definitions of limit operators, characterizations of compactness and Fredholmness using limit operators; and thus also spectral consequences. In this way, we recover and unify the classical and recent results on limit operator techniques, but also gain new insights and are able to treat further applications.
  
  \medskip

  MSC2010: Primary: 47A53; Secondary: 47B07, 47B35-38, 47L10

  Key words: metric measure spaces, bounded geometry, limit operators, Fredholm, essential spectrum%
\end{abstract}

\section{Introduction}

One method to study Fredholm properties of operators are so-called limit operators. This theory has its roots in \cite{Buck1958, Favard1927,LangeRabinovich1985,LangeRabinovich1985b, LangeRabinovich1986,Muhamadiev1971,Muhamadiev1972,Muhamadiev1984I, Muhamadiev1984II,RabinovichRochSilbermann1998, RabinovichRochSilbermann2001}, see also the monographs \cite{Lindner2006,RabinovichRochSilbermann2004} for a thorough treatment.
In a nutshell, it deals with the following. Given an operator $A$ on some $L_p$-space over a metric measure space $(X,d,\mu)$, limit operators can be thought of limits of shifted copies of $A$ as the shifts tend to the boundary of $X$ (or to infinity with respect to the metric $d$). One then aims to obtain information on $A$ by studying its limit operators $A_x$. One typical statement in this direction is: 

\medskip

\begin{enumerate}
  \item[(I)] $A$ is compact if and only if all limit operators of $A$ are trivial.
\end{enumerate}

\medskip

\noindent
Since an operator is Fredholm if and only if it is invertible modulo compact operators, another statement to expect is:

\medskip

\begin{enumerate}
  \item[(II)] $A$ is Fredholm if and only if all limit operators are invertible.
\end{enumerate}

\medskip

\noindent
Since the essential spectrum of an operator is related to Fredholmness, as a consequence one then obtains:

\medskip

\begin{enumerate}
  \item[(III)] The essential spectrum of $A$ coincides with the union of the spectra of its limit operators.
\end{enumerate}

\medskip

Classical results in this direction dealt with operators on scalar-valued $\ell_2$-spaces over $\Z$ or $\Z^n$ and its generalizations to $\ell_p$ for the reflexive range $1<p<\infty$ \cite{LangeRabinovich1985}. To treat the the endpoint cases $p\in\set{1,\infty}$ and vector-valued analogs one needs a further ingredient, the so-called $\mathcal{P}$-theory \cite{Lindner2006,RabinovichRochSilbermann2004,Seidel2014}. Although the method of limit operators can be considered to be classical, surprisingly, significant progress towards statement (II) was made only recently. Even more surprisingly, the result was first shown for $p \in \set{1,\infty}$ (see \cite{ChandlerWildeLindner2011, Lindner2003,Lindner2006,RabinovichRochSilbermann1998}) and then generalized to the Wiener algebra for all $p \in [1,\infty]$ by an interpolation argument \cite{RabinovichRochSilbermann2004}. The much more difficult case of band-dominated operators was then solved by Lindner and Seidel \cite{LindnerSeidel2014} in 2014. Before these recent developments, even for $\ell_2(\Z)$, there was a somewhat ``nasty'' uniform invertibility condition involved, which was difficult to deal with in applications.

In the last few years, limit operator techniques have been applied to a variety of situations such as operators on Fock spaces \cite{FulscheHagger2019}, on Bergman spaces \cite{Hagger2017, Hagger2019}, and also on uniformly discrete metric measure spaces of bounded geometry \cite{SpakulaWillett2017,Zhang2018}. Further applications to spectral theory for operator families can be found in \cite{BeckusLenzLindnerSeifert2017, BeckusLenzLindnerSeifert2018}. Although they all share the common method it appears that the corresponding techniques are tailor-made for the particular situation. This is the starting point for our paper. The aim of this paper is to single out the assumptions needed to obtain results of type (I), (II), and (III) in the abstract framework of $L_p$-spaces $L_p(X,\mu)$ for metric measure spaces $(X,d,\mu)$ and $1<p<\infty$. In this way, we unify the theory for the different applications in the literature as well as gain new insights into the method.
Specifically, we show that there is an interplay between geometric properties of the metric measure space $X$, compactness of subsets of $X$ and compact operators, abstract shifts on $X$ and on $L_p(X,\mu)$ and compactifications of $X$.

We also want to mention that for the Hilbert space case (i.e.~$p=2$) there are $C^*$-algebras techniques available to treat (I), (II) and (III) which are not used in this paper, see e.g.\ \cite{CarvalhoNistorQiao2017, CarvalhoNistorQiao2018,Georgescu2010,Georgescu2018,Roe2005}. They lead to similar connections with coarse geometry as the methods presented here.

The paper is organised as follows. In Section \ref{sec:propA} we review proper metric spaces of bounded geometry and introduce property \Aprime{}, which will be needed in the sequel. The name of our property \Aprime{} is of course inspired by Yu's property \A{} \cite{Yu2000}, which implies property \Aprime{}; in fact they are equivalent for our metric spaces. $X$ also satisfies property \Aprime{} if it has the much more convenient finite asymptotic dimension. Band-dominated operators are then introduced in Section \ref{sec:BDO}. We characterize the space of band-dominated operators and show some algebraic properties of it. Further, we introduce Toeplitz operators in our setting, which is a source for many applications. The core of this paper is Section \ref{sec:limit_operators}, where we state the assumptions for which we show that the limit operator method yields a rich theory. Here we only work with the Stone-\v{C}ech compactification $\beta X$ of $X$ and its boundary $\bdry{X}:=\beta X\setminus X$. After introducing the limit operators we focus on compact operators and characterize them as in (I); cf.\ Corollary \ref{cor:characterization_compact_operators_limit_operators}. Then we aim at proving (II) in Theorem \ref{thm:Fredholm_equivalence} and, as a consequence, (III) in Corollary \ref{cor:spectral_equality}. In order to do this, we need the notion of lower norms of an operator and its properties. Section \ref{sec:compactifications} is devoted to the influence of the compactification of $X$ we use. Since we worked with the ``largest one'' in the previous section, this is now an easy consequence of the universal property of the Stone-\v{C}ech compactification. The last Section \ref{sec:applications} collects various applications. Here, we recover many statements on limit operators in various settings by just adjusting our abstract framework. Moreover, we present further examples not yet treated in the literature so far.

\medskip

Although we only consider scalar-valued $L_p$-spaces and the reflexive range $1<p<\infty$, this is mainly for convenience and readability. We expect to be able to treat the cases $p\in\set{1,\infty}$ as well as Banach space-valued Bochner-Lebesgue spaces $L_p(X,\mu;Y)$ (where $Y$ is a Banach space) using the $\mathcal{P}$-theory. This will be investigated in a future project. However, the vector-valued case (i.e.~if $Y$ is finite-dimensional) can already be treated with the methods presented here using a simple tensoring trick, which is sketched in Section \ref{sec:applications}.

We end this introduction with some notation used in the remaining part. For Banach spaces $E$ and $F$ we write $\calL(E,F)$ for the bounded linear operators and $\calL(E):=\calL(E,E)$. The ideal of compact operators is denoted by $\calK(E,F)$ and $\calK(E)$, respectively. For a measurable subset $K$ of a measure space we write $\1_K$ for the indicator function of $K$ (which is exactly one on $K$ and zero elsewhere). Operators of multiplication by measurable functions $f$ are denoted by $M_f$. In a metric space $(X,d)$, we denote by $B(x,r)$ the open ball around $x\in X$ of radius $r>0$. Correspondingly, $B[x,r]$ denotes the closed ball. We write $\abs{\cdot}$ for the cardinality function of sets. Norms of vectors and operators will be denoted by $\|\cdot\|$ and supplemented with an appropriate subscript if necessary. The complement of a set $Y \subseteq X$ will be denoted by $Y^c$.

\section{On bounded geometry and Property \A}
\label{sec:propA}

Let $X$ be a metric space with metric $d$.
$X$ is called of \emph{bounded geometry} if there exists $\varepsilon > 0$ such that for all $r>0$ there exists $N_r\in\N$ such that for all $x\in X$ the ball $B(x,r)$ can be covered by at most $N_r$ open balls of radius $\varepsilon$. In other words, $X$ is of bounded geometry if every open ball $B(x,r)$ can be covered by a finite number of balls of radius $\varepsilon$ and this finite number $N_r$ only depends on $r$ and not on $x$.

$X$ is called \emph{proper} provided that closed balls are compact. Proper metric spaces are locally compact (hence the Riesz-Markov-Kakutani Theorem applies). Proper metric spaces are also complete, $\sigma$-compact and hence separable.

\begin{lemma}
\label{lem:covering}
  Let $(X,d)$ be a separable metric space of bounded geometry (with some $\varepsilon > 0$) and $r \geq \varepsilon$. Then there exists $J \subseteq \N$ and a sequence $(Q_j)_{j\in J}$ of Borel subsets of $X$ such that
  \begin{enumerate}
    \item
      $Q_j \cap Q_k = \varnothing$ for $j,k\in J$, $j\neq k$,
    \item
      $X = \bigcup\limits_{j\in J} Q_j$,
    \item
      $\diam Q_j\leq 4r$ for all $j\in J$,
    \item
      there is a positive integer $N$ such that for all $k \in J$ the set $J_k(r) := \set{j \in J : \dist(Q_j,Q_k) \leq r}$ has at most $N$ elements,
    \item
      for all $x\in X$ and $s>0$ the set $\set{j\in J: Q_j\cap B(x,s) \neq \varnothing}$ is finite.
  \end{enumerate}
\end{lemma}

Here, $\diam$ denotes the diameter of a set, i.e.~$\diam(Q) := \sup\limits_{x,y \in Q} d(x,y)$, and $\dist$ denotes the distance between two sets, i.e.~$\dist(Q_1,Q_2) := \inf\limits_{x \in Q_1,y \in Q_2} d(x,y)$. We will use the abbreviation $\dist(x,Q) := \dist(\set{x},Q)$ for the distance between a point $x$ and a set $Q$.

\begin{proof}
  For $r>0$ there exists $J\subseteq \N$ and $(x_j)_{j\in J}$ such that $(x_j)_{j\in J}$ is a maximal $2r$-separated sequence in $X$, i.e.\ $d(x_j,x_k) \geq 2r$ for all $j,k \in J$, $j \neq k$, and for all $x \in X$ there exists $j \in J$ such that $d(x,x_j) < 2r$. Without loss of generality let $J = \set{1,\ldots,\abs{J}}$ in case $J$ is finite and $J = \N$ in case $J$ is infinite.
  
  Define $A_j:=\bigcup\limits_{k\in J, k \neq j} B(x_k,r)$ for $j\in J$. Then $A_j\cap B(x_j,r) = \varnothing$ for all $j\in J$. Now define $Q_1 := B(x_1,2r) \setminus A_1$ and
  \[Q_j:= \Biggl(B(x_{j},2r) \setminus \Bigl(\bigcup_{k=1}^{j-1} Q_k\Bigr)\Biggr)\setminus A_j = B(x_j,2r)\setminus \Bigl(\bigcup_{k=1}^{j-1} Q_k \cup A_j\Bigr)\]
  for $j \geq 2$. Then the sequence $(Q_j)_{j \in J}$ satisfies (a), (b) and (c). Moreover, we have $B(x_j,r) \subseteq Q_j \subseteq B(x_j,2r)$ for all $j\in J$. Therefore, for (d) it suffices to show that there is a constant $N \in \N$ such that $\abs{\set{j \in J : x_j \in B(x_k,6r)}} \leq N$ for all $k \in J$.
  
  By the bounded geometry assumption there is an $\varepsilon > 0$ and an integer $N:= N_{6r}$ such that every ball $B(x_k,6r)$ can be covered by at most $N$ balls of radius $\varepsilon$, i.e.~for every $k \in J$ there exist $y_1, \ldots, y_N \in X$ such that $B(x_k,6r) \subseteq \bigcup\limits_{l = 1}^N B(y_l,\varepsilon)$. Now as $d(x_j,x_k) \geq 2r$ for all $j,k\in J$ with $j \neq k$ and $\diam(B(y_l,\varepsilon)) = 2\varepsilon \leq 2r$, every $B(y_l,\varepsilon)$ can only contain at most one $x_j$. Hence $\abs{\set{j \in J : x_j \in B(x_k,6r)}} \leq N$ for all $k \in J$.
  
  To show the last property, let $j\in J$ such that $Q_j\cap B(x,s) \neq \varnothing$. Since $Q_j\subseteq B(x_j,2r)$ we observe $d(x,x_j)<2r+s$. Hence, if also $k\in J$ such that $Q_k\cap B(x,s) \neq \varnothing$, then $d(x_j,x_k)\leq d(x_j,x) + d(x,x_k) < 4r+2s$.
  We now reason as above. Let $N:=N_{4r+2s}$ and $y_1,\ldots,y_N\in X$ such that $B(x_j,4r+2s) \subseteq \bigcup_{l=1}^N B(y_l,\varepsilon)$. Since $d(x_j,x_k)\geq 2r$ and $\diam(B(y_l,\varepsilon)) = 2\varepsilon \leq 2r$ for all $l\in\set{1,\ldots,N}$, we conclude that every $B(y_l,\varepsilon)$ can contain at most one $x_k$. Hence
  $\abs{\set{k \in J : x_k \in B(x_j,4r+2s)}} \leq N$, and therefore $\abs{\set{j\in J: Q_j\cap B(x,s) \neq \varnothing}}\leq N$.
\end{proof}

Let $Y$ be a Banach space, $\xi\from X\to Y$ a map and $R,\varepsilon>0$. We say that $\xi$ has \emph{$(R,\varepsilon)$-variation} if for all $x,y\in X$ satisfying $d(x,y)\leq R$ we have $\norm{\xi(x)-\xi(y)}\leq \varepsilon$.

\begin{definition}[{\cite[Definition 5.2.2]{Willett2009}}] \label{def:property_A}
  Let $X$ be a proper metric space of bounded geometry.
  Then $X$ has \emph{property \A{}} provided for all $R,\varepsilon>0$ there exists a weak$^*$-continuous map $\mu\from X\to C_0(X)^* = \mathcal{M}(X)$ (the regular countably additive Borel measures on $X$) 
  such that
  \begin{enumerate}
    \item $\norm{\mu_x} = 1$ for all $x\in X$,
    \item $\mu$ has $(R,\varepsilon)$-variation,
    \item there exists $S>0$ such that for all $x\in X$ the functional (=complex Radon measure) $\mu_x$ is supported in $B[x,S]$.
  \end{enumerate}
\end{definition}

\begin{definition} \label{def:property_A'}
  Let $(X,d)$ be a proper metric space of bounded geometry. We say that $X$ has \emph{property \Aprime{}} if for every $t > 0$ there is a countably infinite collection $(\varrho_{j,t})_{j\in\N}$ of non-zero measurable functions $\varrho_{j,t} \from X \to [0,1]$ ($j\in\N$) such that
    \begin{itemize}
	  \item[(i)] $\sum\limits_{j \in \N} \varrho_{j,t}(x) = 1$ for all $x \in X$,
	  \item[(ii)] $\sup\limits_{j \in \N} \diam(\spt \varrho_{j,t}) < \infty$,
	  \item[(iii)] $d(x,y) \leq \frac{1}{t}$ implies $\sum\limits_{j \in \N} \abs{\varrho_{j,t}(x) - \varrho_{j,t}(y)} < t$,
	  \item[(iv)] for all $x\in X$ and $s>0$ the set $\set{j \in \N : \spt \varrho_{j,t} \cap B(x,s) \neq \varnothing}$ is finite.
  \end{itemize}
\end{definition}

Here, $\spt$ denotes the support of a function or a measure, where the support for a measure is defined by duality with $C_0(X)$.

% \begin{proposition} \label{prop:A_implies_A'}
%   Let $X$ be an unbounded proper metric space of bounded geometry. Then property \A{} implies property \Aprime.
% \end{proposition}

\begin{theorem} \label{thm:A_equiv_A'}
  Let $X$ be an unbounded proper metric space of bounded geometry. Then property \A{} is equivalent to property \Aprime.
\end{theorem}

\begin{proof}
  Assume that $X$ has property \A{} and fix $t > 0$. Let $\mu$ be as in Definition \ref{def:property_A} with $(\frac{1}{t},\frac{t}{2})$-variation. Without loss of generality we may assume that $\mu_x\geq 0$ for all $x\in X$ (otherwise consider $\abs{\mu_x}$).
  
  Choose $r > 0$ and choose a sequence $(Q_j)_{j \in \N}$ according to Lemma \ref{lem:covering} (as $X$ is unbounded, we must have $J \cong \N$). For $\delta > 0$ let $Q_{j,\delta} := \set{x \in X : \dist(x,Q_j) < \delta}$. Then $\mathcal{U} := \set{Q_{j,\delta} : j \in \N}$ is a countable open covering of $X$. Let $(\psi_j)$ be a partition of unity subordinated to $\mathcal{U}$. In particular, we have $\spt \psi_j \subseteq Q_{j,\delta}$ for all $j \in \N$. We define $\varrho_{j,t} \from X \to [0,1]$ by the duality pairing $\varrho_{j,t}(z):=\dupa{\mu_z}{\psi_j}$. Then
  $\varrho_{j,t}$ is continuous, hence measurable. Moreover,
  \[\sum_{j \in \N} \varrho_{j,t}(z) = \sum_{j \in \N} \dupa{\mu_z}{\psi_{j,t}} = \dupa{\mu_z}{\sum_{j \in \N} \psi_{j,t}} = \dupa{\mu_z}{\1} = 1\]
  by monotone convergence and the partition of unity property. Therefore we have (i).
  
  Let $j \in \N$. By property (c) of Defintion \ref{def:property_A}, there exists $S > 0$ such that $\mu_z$ is supported in $B[z,S]$ for all $z \in X$. If $z \in X$ with $\dist(z,Q_{j,\delta}) > S$, then $\spt \mu_z \cap Q_{j,\delta} = \varnothing$, so $\varrho_{j,t}(z) = 0$. Hence, $\spt \varrho_{j,t} \subseteq Q_{j,\delta+S}$. We get
  \[\diam(\spt \varrho_{j,t}) \leq \diam(Q_j)+2\delta+2S \leq 4r+2\delta+2S\]
  and (ii) follows.
  
  Let $z,z'\in X$ with $d(z,z')\leq \frac{1}{t}$. Then, by monotone convergence and since $(\psi_j)_{j \in \N}$ is a partition of unity,
  \begin{align*}
  \sum_{j \in \N} \abs{\varrho_{j,t}(z) - \varrho_{j,t}(z')} \leq \sum_{j \in \N} \dupa{\abs{\mu_z - \mu_{z'}}}{\psi_j} = \dupa{\abs{\mu_z - \mu_{z'}}}{\sum_{j \in \N} \psi_j} = \dupa{\abs{\mu_z - \mu_{z'}}}{\1} = \norm{\mu_z-\mu_{z'}} \leq \frac{t}{2} < t,
  \end{align*}
  which yields (iii).
  
  Let $x \in X$ and $s > 0$. As $\set{j \in \N: Q_j \cap B(x,s+\delta+S) \neq \varnothing}$ is finite by Lemma \ref{lem:covering}(e) and $\spt \varrho_{j,t} \subseteq Q_{j,\delta+S}$, (iv) follows as well.
  
  Finally, as $X$ is assumed to be unbounded, it is clear that we need infinitely many non-zero $\varrho_{j,t}$ for a partition of unity and therefore we may just delete all $\varrho_{j,t}$ that are zero.
  
  \medskip
  
  Now assume $(X,d)$ has property \Aprime. 
  Let $\varepsilon>0$ as in the definition of bounded geometry of $X$. Let $X_0\subseteq X$ be a maximal $2\varepsilon$-separated subset of $X$ and $d_0:=d|_{X_0\times X_0}$.
  Then $(X_0,d_0)$ is countable (and infinite since $d$ is unbounded) and discrete. Moreover, $X=\bigcup_{x\in X_0} B_X(x,2\varepsilon)$ which means that $X_0$ is coarsely dense in $X$.
  By the bounded geometry of $X$, for $r>0$ there exists $N_r\in\N$ such that for all $x\in X_0$ the ball $B_{X_0}(x,r)\subseteq B_X(x,r)$ can be covered by at most $N_r$ open balls (in $X$) of radius $\varepsilon$. Since $X_0$ is $2\varepsilon$-separated, each of these $\varepsilon$-balls in $X$ can contain at most one point of $X_0$. Thus, $B_{X_0}(x,r)$ has at most $N_r$ elements. Put differently, $(X_0,d_0)$ has bounded geometry `with $\varepsilon=0$'.
  Let $R,\varepsilon>0$. Let $0<t<\min\set{\tfrac{1}{R},\varepsilon}$ and $(\varrho_{j,t})_{j\in\N}$ as in property \Aprime, $S:=\sup\limits_{j \in \N} \diam(\spt \varrho_{j,t}) < \infty$.
  Set $U_j:=\spt \varrho_{j,t}$ ($j\in\N$). Then $(U_j)_{j\in \N}$ is a cover of $X_0$, and $(\varrho_{j,t})_{j\in\N}$ is a partition of unity of $X_0$ subordinated to $(U_j)$.
  Let $x,y\in X_0$ with $d(x,y)\leq R$. Then $d(x,y)\leq \tfrac{1}{t}$ and therefore $\sum_{j\in\N} \abs{\varrho_{j,t}(x)-\varrho_{j,t}(y)}< t\leq \varepsilon$.
  Moreover, $\diam U_j = \diam \spt \varrho_{j,t}\leq S$ for all $j\in\N$.
  Now, \cite[Theorem 1.2.4]{Willett2009} yields that $X_0$ has property \A{}. Since $X_0$ is coarsely dense in $X$, \cite[Lemma 5.2.4]{Willett2009} yields that $X$ has property \A{}.  
\end{proof}

% For discrete metric spaces with the property that for all $r > 0$ there exists $N_r \in \N$ such that $|B(x,r)| \leq N_r$ for all $x \in X$ (i.e.~bounded geometry with `$\varepsilon = 0$') property \A{} and property \Aprime{} are equivalent (see \cite[Theorem 1.2.4]{Willett2009}).
% We do not know whether this still holds true for non-discrete metric spaces. Another related notion is the asymptotic dimension. 
Another related notion is the asymptotic dimension. 
A metric space has finite asymptotic dimension if there exists $N \in \N$ such that for all $r < \infty$ there is a uniformly bounded open cover with $r$-multiplicity less than $N$. Here, $r$-multiplicity means that every open ball of radius $r$ intersects with at most $N$ sets from the cover (cf.~\cite[Section 3]{BellDranishnikov2008}).

\begin{proposition} \label{prop:asymp_dim}
Let $(X,d)$ be a proper metric space of bounded geometry. If $X$ has finite asymptotic dimension, then $X$ also has property \Aprime{} (and therefore also property \A).
\end{proposition}

\begin{proof}
Assuming bounded geometry, the construction on page 6 of \cite{Hagger2017} does the job.
\end{proof}

For completeness we mention that a similar construction as in \cite{Hagger2017} was already used in \cite[Corollary 2.2.11]{Willett2009} to show that finite asymptotic dimension implies property \A{} in the discrete setting. %As for property \A{} we do not know whether finite asymptotic dimension and property \Aprime{} are equivalent.
Note that there are proper metric spaces of bounded geometry with property \A{} (and hence property \Aprime) but infinite asymptotic dimension, see e.g.~\cite[Section 23]{BellDranishnikov2008}.

\section{Band-dominated operators on metric spaces}
\label{sec:BDO}

Let $(X,d)$ be a proper metric space of bounded geometry that satisfies property \Aprime, $\mu$ a Borel measure on $X$, $p\in (1,\infty)$. For every $t > 0$ we fix a family of functions $(\varrho_{j,t})_{j\in\N}$ that satisfies the axioms in Definition \ref{def:property_A'}. We will use these functions for the rest of the paper.

\begin{lemma} \label{lem:phis}
Let $(X,d)$ be a proper metric space of bounded geometry that satisfies property \Aprime{} and let $\varphi_{j,t} := \varrho_{j,t}^{1/p}$ for $p \geq 1$, $j \in \N$ and $t > 0$. Then
\begin{itemize}
	  \item[(i)] $\sum\limits_{j \in \N} \left[\varphi_{j,t}(x)\right]^p = 1$ for all $x \in X$,
	  \item[(ii)] $\sup\limits_{j \in \N} \diam(\spt \varphi_{j,t}) < \infty$,
	  \item[(iii)] $d(x,y) \leq \frac{1}{t}$ implies $\sum\limits_{j \in \N} \abs{\varphi_{j,t}(x) - \varphi_{j,t}(y)}^p < t$,
	  \item[(iv)] for all $x\in X$ and $s>0$ the set $\set{j \in \N : \spt \varphi_{j,t} \cap B(x,s) \neq \varnothing}$ is finite.
  \end{itemize}
\end{lemma}

\begin{proof}
(i), (ii) and (iv) are clear. For (iii) use $\abs{a^{1/p} - b^{1/p}} \leq \abs{a-b}^{1/p}$ for $a,b \geq 0$.
\end{proof}

\begin{definition}
  Let $A\in \calL\bigl(L_p(X,\mu)\bigr)$. We call
  \[\prop(A):= \sup\set{\dist(K,K'):\; K,K'\subseteq X,\, M_{\1_{K'}}AM_{\1_K}\neq 0} \in [0,\infty]\]
  the \emph{propagation} or \emph{band width} of $A$.
  
  Let $\BO:=\set{A\in \calL\bigl(L_p(X,\mu)\bigr):\; \prop(A) < \infty}$ denote the set of \emph{operators of finite propagation} or \emph{band operators}. Its norm closure
  $\BDO^p:=\overline{\BO}\subseteq \calL\bigl(L_p(X,\mu)\bigr)$ is called the set of \emph{band-dominated operators}.
\end{definition}

In this section we aim for various properties and a crucial characterization of $\BDO^p$. We start with a technical lemma about band operators.

\begin{lemma}
\label{lem:BO_commutator}
  Let $(X,d)$ be a proper metric space of bounded geometry (with some $\varepsilon > 0$) that satisfies property \Aprime{} and let $\omega \geq 0$. Then there is a constant $C \geq 0$ such that for every $t < \frac{1}{9\max\{\omega,\varepsilon\}}$, every $f \in L_p(X,\mu)$ and all band operators $A$ with $\prop(A) \leq \omega$ the estimate
  \[\Biggl(\sum\limits_{j = 1}^{\infty} \norm{[A,M_{\varphi_{j,t}}]f}_p^p\Biggr)^{1/p} \leq C\norm{A}\norm{f}_pt^{1/p}\]
  holds.
\end{lemma}

Here $[\cdot,\cdot]$ denotes the commutator of two operators, i.e.~$[A,B] = AB-BA$.

\begin{proof}
  Let $A$ be a band operator with $\prop(A) \leq \omega$ and $f \in L_p(X,\mu)$. Let $(Q_i)_{i \in I}$ be the sequence of Borel sets coming from Lemma \ref{lem:covering} with $r = \max\{\omega,\varepsilon\}$. For every $i \in I$ choose a point $x_i \in Q_i$. Now, define
  \[\tilde{\varphi}_{j,t} := \sum\limits_{i \in I} \varphi_{j,t}(x_i) \1_{Q_i}.\]
  As $\diam(Q_i) \leq 4r < \frac{1}{t}$, Lemma \ref{lem:phis} implies
  \[\sum\limits_{j = 1}^{\infty} \abs{\varphi_{j,t}(x) - \tilde{\varphi}_{j,t}(x)}^p = \sum\limits_{j = 1}^{\infty} \abs{\varphi_{j,t}(x) - \varphi_{j,t}(x_i)}^p < t\]
  for every $x \in Q_i$ and every $i \in I$, hence every $x \in X$. This implies
  \[\sum\limits_{j = 1}^{\infty} \norm{\left(M_{\varphi_{j,t}} - M_{\tilde{\varphi}_{j,t}}\right)f}_p^p =  \sum\limits_{j = 1}^{\infty} \int_X \abs{\varphi_{j,t}(x) - \tilde{\varphi}_{j,t}(x)}^p\abs{f(x)}^p \, \mathrm{d}\mu(x) < t\norm{f}_p^p\]
  by the monotone convergence theorem. Using the triangular inequality and Minkowski's inequality for sums, we get
  \begin{align*}
  \Biggl(\sum\limits_{j = 1}^{\infty} \norm{[A,M_{\varphi_{j,t}}]f}_p^p\Biggr)^{1/p} &\leq \Biggl(\sum\limits_{j = 1}^{\infty} \norm{[A,M_{\tilde{\varphi}_{j,t}}]f}_p^p\Biggr)^{1/p} + \Biggl(\sum\limits_{j = 1}^{\infty} \norm{[A,M_{\varphi_{j,t}} - M_{\tilde{\varphi}_{j,t}}]f}_p^p\Biggr)^{1/p}\\
  &< \Biggl(\sum\limits_{j = 1}^{\infty} \norm{[A,M_{\tilde{\varphi}_{j,t}}]f}_p^p\Biggr)^{1/p} + 2\norm{A}\norm{f}_pt^{1/p}.
  \end{align*}
  It thus remains to estimate $\sum\limits_{j = 1}^{\infty} \norm{[A,M_{\tilde{\varphi}_j}]f}_p^p$. For this we define the relation
  \[k \sim l :\Longleftrightarrow \dist(Q_k,Q_l) \leq \omega.\]
  Note that by Lemma \ref{lem:covering}, $|\set{l \in I : k \sim l}| \leq N$ and that $N$ only depends on $\omega$ (not on $t$ or $j$).
  
  As $\prop(A) \leq \omega$, we have
  \begin{align*}
  AM_{\tilde{\varphi}_{j,t}} - M_{\tilde{\varphi}_{j,t}}A &= \sum\limits_{k \in I} \varphi_{j,t}(x_k)(AM_{\1_{Q_k}} - M_{\1_{Q_k}}A) = \sum\limits_{\genfrac{}{}{0pt}{}{k,l \in I}{k \sim l}} \varphi_{j,t}(x_k)(M_{\1_{Q_l}}AM_{\1_{Q_k}} - M_{\1_{Q_k}}AM_{\1_{Q_l}})\\
  &= \sum\limits_{\genfrac{}{}{0pt}{}{k,l \in I}{k \sim l}} (\varphi_{j,t}(x_k) - \varphi_{j,t}(x_l))M_{\1_{Q_l}}AM_{\1_{Q_k}}.
  \end{align*}
  Using the usual dual pairing $\dupa{\cdot}{\cdot}$ ($\frac{1}{p} + \frac{1}{q} = 1$) and H\"older's inequality twice, we get
  \begin{align*}
  \norm{\sum\limits_{\genfrac{}{}{0pt}{}{k,l \in I}{k \sim l}} (\varphi_{j,t}(x_k) - \varphi_{j,t}(x_l))M_{\1_{Q_l}}AM_{\1_{Q_k}}f}_p &= \sup\limits_{\norm{g}_q = 1} \abs{\dupa{\sum\limits_{\genfrac{}{}{0pt}{}{k,l \in I}{k \sim l}} (\varphi_{j,t}(x_k) - \varphi_{j,t}(x_l))M_{\1_{Q_l}}AM_{\1_{Q_k}}f}{g}}\\
  &\leq \sup\limits_{\norm{g}_q = 1} \sum\limits_{\genfrac{}{}{0pt}{}{k,l \in I}{k \sim l}} \abs{\varphi_{j,t}(x_k) - \varphi_{j,t}(x_l)} \abs{\dupa{AM_{\1_{Q_k}}f}{M_{\1_{Q_l}}g}}\\
  &\leq \sup\limits_{\norm{g}_q = 1} \sum\limits_{\genfrac{}{}{0pt}{}{k,l \in I}{k \sim l}} \abs{\varphi_{j,t}(x_k) - \varphi_{j,t}(x_l)} \norm{AM_{\1_{Q_k}}f}_p\norm{M_{\1_{Q_l}}g}_q\\
  &\leq \sup\limits_{\norm{g}_q = 1} \Biggl(\sum\limits_{\genfrac{}{}{0pt}{}{k,l \in I}{k \sim l}} \abs{\varphi_{j,t}(x_k) - \varphi_{j,t}(x_l)}^p \norm{AM_{\1_{Q_k}}f}_p^p\Biggr)^{1/p}\\
  &\qquad \qquad \cdot \Biggl(\sum\limits_{\genfrac{}{}{0pt}{}{k,l \in I}{k \sim l}} \norm{M_{\1_{Q_l}}g}_q^q\Biggr)^{1/q}\\
  &\leq \Biggl(\sum\limits_{\genfrac{}{}{0pt}{}{k,l \in I}{k \sim l}} \abs{\varphi_{j,t}(x_k) - \varphi_{j,t}(x_l)}^p \norm{AM_{\1_{Q_k}}f}_p^p\Biggr)^{1/p}N^{1/q}
  \end{align*}
  since the $Q_k$ are pairwise distinct and $|\set{l \in I : k \sim l}| \leq N$, i.e.~every $Q_k$ is counted at most $N$ times. For $k \sim l$ we have $d(x_k,x_l) \leq \diam(Q_k) + \dist(Q_k,Q_l) + \diam(Q_l) \leq 9r < \frac{1}{t}$ and hence
  \begin{align*}
  \sum\limits_{j = 1}^{\infty} \norm{[A,M_{\tilde{\varphi}_j}]f}_p^p &= \sum\limits_{j = 1}^{\infty} \norm{\sum\limits_{\genfrac{}{}{0pt}{}{k,l \in I}{k \sim l}} (\varphi_{j,t}(x_k) - \varphi_{j,t}(x_l))M_{\1_{Q_l}}AM_{\1_{Q_k}}f}_p^p\\
  &\leq \sum\limits_{j = 1}^{\infty}\sum\limits_{\genfrac{}{}{0pt}{}{k,l \in I}{k \sim l}} \abs{\varphi_{j,t}(x_k) - \varphi_{j,t}(x_l)}^p \norm{AM_{\1_{Q_k}}f}_p^pN^{p/q}\\
  &= N^{p/q}\sum\limits_{\genfrac{}{}{0pt}{}{k,l \in I}{k \sim l}} \norm{AM_{\1_{Q_k}}f}_p^p \sum\limits_{j = 1}^{\infty} \abs{\varphi_{j,t}(x_k) - \varphi_{j,t}(x_l)}^p\\
  &< N^{p/q}\sum\limits_{\genfrac{}{}{0pt}{}{k,l \in I}{k \sim l}} \norm{AM_{\1_{Q_k}}f}_p^p t\\
  &< N^p\norm{A}^p\norm{f}_p^pt,
  \end{align*}
  by the same arguments as above.
\end{proof}

For the characterization of band-dominated operators we will need another auxiliary lemma.

\begin{lemma}
\label{lem:sum_estimate}
Let $f \in L_p(X,\mu)$, $(A_j)_{j \in \N}$ in $\calL(L_p(X,\mu))$, $t>0$, and $s \geq p-1$. Then
\[\Biggl\|\sum\limits_{j \in \N} M_{\varphi_{j,t}^s}A_jf\Biggr\|_p \leq \Biggl(\sum\limits_{j \in \N} \norm{M_{\varphi_{j,t}^{s-p+1}}A_jf}_p^p\Biggr)^{1/p}.\]
\end{lemma}

\begin{proof}
Let $\frac{1}{p} + \frac{1}{q} = 1$. Using the usual dual pairing $\dupa{\cdot}{\cdot}$ and H\"older's inequality twice, we get
  \begin{align*}
  \Biggl\|\sum\limits_{j \in \N} M_{\varphi_{j,t}^s}A_jf\Biggr\|_p &= \sup\limits_{\norm{g}_q = 1} \Biggl|\Biggl\langle \sum\limits_{j \in \N} M_{\varphi_{j,t}^s}A_jf,g\Biggr\rangle\Biggr|\\
  &= \sup\limits_{\norm{g}_q = 1} \Biggl|\sum\limits_{j \in \N} \dupa{M_{\varphi_{j,t}^{s-p+1}}A_jf}{M_{\varphi_{j,t}^{p/q}}g}\Biggr|\\
  &\leq \sup\limits_{\norm{g}_q = 1} \sum\limits_{j \in \N} \norm{M_{\varphi_{j,t}^{s-p+1}}A_jf}_p\norm{M_{\varphi_{j,t}^{p/q}}g}_q\\
  &\leq \sup\limits_{\norm{g}_q = 1} \Biggl(\sum\limits_{j \in \N} \norm{M_{\varphi_{j,t}^{s-p+1}}A_jf}_p^p\Biggr)^{1/p} \Biggl(\sum\limits_{j \in \N} \norm{M_{\varphi_{j,t}^{p/q}}g}_q^q\Biggr)^{1/q}\\
  &= \Biggl(\sum\limits_{j \in \N} \norm{M_{\varphi_{j,t}^{s-p+1}}A_jf}_p^p\Biggr)^{1/p}
  \end{align*}
  since $\sum\limits_{j \in \N} \varphi_{j,t}^p = 1$.
\end{proof}

The following Proposition is an adaptation of \cite[Theorem 2.1.6]{RabinovichRochSilbermann2004} to our setting. More precisely, we characterise an operator $A$ to be band-dominated by looking at (limits of) suitable commutator estimates of $A$ with multiplication operators.

\begin{proposition}
\label{prop:BDO_commutator}
  Let $(X,d)$ be a proper metric space of bounded geometry that satisfies property \Aprime{} and $A\in \calL\bigl(L_p(X,\mu)\bigr)$. Then the following are equivalent:
  \begin{enumerate}
    \item
      $A\in \BDO^p$.
    \item
      $\lim\limits_{t\to 0} \sup\limits_{\norm{f}_p = 1} \sum\limits_{j = 1}^{\infty} \norm{[A,M_{\varphi_{j,t}}]f}_p^p = 0$.
  \end{enumerate}
\end{proposition}

\begin{proof}
  ``(a)$\Rightarrow$(b)'':
  Let $\varepsilon > 0$ and choose $A_0 \in \BO$ such that $\norm{A-A_0} \leq \varepsilon$. Then, by Lemma \ref{lem:BO_commutator}, there is a $t_0 > 0$ such that $\sup\limits_{\norm{f}_p = 1} \sum\limits_{j = 1}^{\infty} \norm{[A_0,M_{\varphi_{j,t}}]f}_p^p \leq \varepsilon^p$ for all $t < t_0$. It follows
  \begin{align*}
  &\sup\limits_{\norm{f}_p = 1} \Biggl(\sum\limits_{j = 1}^{\infty} \norm{[A,M_{\varphi_{j,t}}]f}_p^p\Biggr)^{1/p}\\
  &\qquad \qquad \qquad \leq \sup\limits_{\norm{f}_p = 1} \Biggl(\sum\limits_{j = 1}^{\infty} \norm{[A_0,M_{\varphi_{j,t}}]f}_p^p\Biggr)^{1/p} + \sup\limits_{\norm{f}_p = 1} \Biggl(\sum\limits_{j = 1}^{\infty} \norm{[A-A_0,M_{\varphi_{j,t}}]f}_p^p\Biggr)^{1/p}\\
   &\qquad \qquad \qquad \leq \varepsilon + \norm{A-A_0}\sup\limits_{\norm{f}_p = 1} \Biggl(\sum\limits_{j = 1}^{\infty} \norm{M_{\varphi_{j,t}}f}_p^p\Biggr)^{1/p} + \sup\limits_{\norm{f}_p = 1} \Biggl(\sum\limits_{j = 1}^{\infty} \norm{M_{\varphi_{j,t}}(A-A_0)f}_p^p\Biggr)^{1/p}\\
   &\qquad \qquad \qquad \leq 3\varepsilon
  \end{align*}
  since $\sum\limits_{j \in \N} \varphi_{j,t}^p = 1$. Hence $\lim\limits_{t\to 0} \sup\limits_{\norm{f}_p = 1} \sum\limits_{j = 1}^{\infty} \norm{[A,M_{\varphi_{j,t}}]f}_p^p = 0$.
  
  ``(b)$\Rightarrow$(a)'':
  Let $A_t := \sum\limits_{j = 1}^{\infty} M_{\varphi_{j,t}^{p/q}}AM_{\varphi_{j,t}}$, where $\frac{1}{p} + \frac{1}{q} = 1$. Then
  \[\norm{A_t - A} = \Biggl\|\sum\limits_{j = 1}^{\infty} M_{\varphi_{j,t}^{p/q}}AM_{\varphi_{j,t}} - A\Biggr\| \leq \Biggl\|\sum\limits_{j = 1}^{\infty} M_{\varphi_{j,t}^{p/q}}M_{\varphi_{j,t}}A - A\Biggr\| + \Biggl\|\sum\limits_{j = 1}^{\infty} M_{\varphi_{j,t}^{p/q}}[A,M_{\varphi_{j,t}}]\Biggr\|.\]
  The first term vanishes since $\sum\limits_{j = 1}^{\infty} \varphi_{j,t}^{p/q}\varphi_{j,t} = \sum\limits_{j = 1}^{\infty} \varphi_{j,t}^p = 1$. Using Lemma \ref{lem:sum_estimate}, we can estimate the second term:
  \[\Biggl\|\sum\limits_{j = 1}^{\infty} M_{\varphi_{j,t}^{p/q}}[A,M_{\varphi_{j,t}}]\Biggr\| = \sup\limits_{\norm{f}_p = 1} \Biggl\|\sum\limits_{j = 1}^{\infty} M_{\varphi_{j,t}^{p/q}}[A,M_{\varphi_{j,t}}]f\Biggr\|_p \leq \sup\limits_{\norm{f}_p = 1} \Biggl(\sum\limits_{j = 1}^{\infty} \norm{[A,M_{\varphi_{j,t}}]f}_p^p\Biggr)^{1/p}.\]
  By assumption, this tends to $0$ as $t \to 0$. Thus $(A_t)_{t > 0}$ is a bounded net and converges to $A$ in norm. That $A_t$ is a band operator follows directly from the fact that $\sup\limits_{j \in \N} \diam(\spt \varphi_{j,t}) < \infty$. Hence $A \in \BDO^p$.
\end{proof}

\begin{corollary}
\label{cor:BDO_commutator}
Let $(X,d)$ be a proper metric space of bounded geometry that satisfies property \Aprime, $A\in \BDO^p$ and $g \in C([0,1])$. Then
\[\lim\limits_{t\to 0} \sup\limits_{\norm{f}_p = 1} \sum\limits_{j = 1}^{\infty} \norm{[A,M_{g \circ \varphi_{j,t}}]f}_p^p = 0.\]
\end{corollary}

\begin{proof}
By Proposition \ref{prop:BDO_commutator} and standard commutator relations this holds for polynomials $p \from [0,1] \to \C$. As $\norm{M_{g \circ \varphi_{j,t}} - M_{p \circ \varphi_{j,t}}} \leq \norm{g - p}_{\infty}$, this generalizes to arbitrary continuous functions $g \in C([0,1])$.
\end{proof}

Next, we show some algebraic properties of the set $\BDO^p$.
These properties are well-known for band-dominated operators on $\ell_p$-spaces over $\Z^N$ (cf. \cite[Propositions 2.1.7-2.1.9]{RabinovichRochSilbermann2004}) and transfer to our situation as well.

\begin{theorem}
\label{thm:BDO_algebraic}
  Let $(X,d)$ be a proper metric space of bounded geometry that satisfies property \Aprime. We have
  \begin{enumerate}
    \item 
      $M_f\in \BDO^p$ for all $f\in L_\infty(X,\mu)$.
    \item
      $\BO$ is an algebra and $\BDO^p$ is a closed subalgebra of $\calL\bigl(L_p(X,\mu)\bigr)$.
    \item
      Let $A\in \BDO^p$ be Fredholm, $B\in \calL\bigl(L_p(X,\mu)\bigr)$ a regulariser for $A$. Then $B\in \BDO^p$. In particular, $\BDO^p$ is inverse closed in $\calL\bigl(L_p(X,\mu)\bigr)$.
    \item
      $\BDO^p$ contains $\calK\bigl(L_p(X,\mu)\bigr)$ as a closed two-sided ideal.
    \item
      Let $A\in\BDO^p$, $\frac{1}{p}+\frac{1}{q}=1$. Then $A^*\in\BDO^q$.
  \end{enumerate}
\end{theorem}

For the proof we will need the following basic fact (see e.g.~\cite[Theorem 1.1.3]{RabinovichRochSilbermann2004} where it is shown for sequences, but the same proof works for bounded nets).

\begin{lemma}
\label{lem:strong_and_compact_convergence}
Let $E$ be a Banach space and $(A_{\iota})$ in $\calL(E)$ be a bounded net of operators converging strongly to $A \in \calL(E)$. Then
\[\norm{A_{\iota}K-AK} \to 0\]
for every $K \in \calK(E)$. Moreover, by taking adjoints, if $(A_{\iota}^*)$ converges strongly to $A^*$, then
\[\norm{KA_{\iota}-KA} \to 0\]
for every $K \in \calK(E)$.
\end{lemma}

We say that a net $(A_{\iota})$ converges \emph{$*$-strongly} if $(A_{\iota})$ and $(A_{\iota}^*)$ both converge strongly.

\begin{proof}[Proof of Theorem \ref{thm:BDO_algebraic}]
  (a)
  Let $f\in L_\infty(X,\mu)$.
  We show that $M_f\in \BO$.
  Let $K,K'\subseteq X$ with $\dist(K,K')>0$. Then $K\cap K' = \varnothing$, so $M_{\1_{K'}}M_{\1_K} = 0$. Since multiplication operators commute, we have
  $M_{\1_{K'}} M_f M_{\1_K} = M_f M_{\1_{K'}} M_{\1_K} = 0$, and therefore $M_f\in \BO$ with $\prop(M_f) = 0$.

  (b)
  Let $A,B\in \BO$. It is clear that $\BO$ is a vector space, hence it suffices to show $AB \in \BO$. Let $K,K'\subseteq X$ such that $\dist(K,K')>\prop(A)+\prop(B)$. Define $K_0:=\set{x\in X:\; \dist(x, K)\leq \prop(B)}$. Then
  $BM_{\1_K} = M_{\1_{K_0}}BM_{\1_K}$. Moreover, $\dist(K',K_0)>\prop(A)$. Hence, $M_{\1_{K'}}ABM_{\1_{K}} = M_{\1_{K'}}AM_{\1_{K_0}}BM_{\1_{K}} = 0$. It follows that $\BO$ is an algebra and $\BDO^p$ is a closed subalgebra of $\calL\bigl(L_p(X,\mu)\bigr)$.
  
  (c)
  Assume that $A \in \calL\bigl(L_p(X,\mu)\bigr)$ is Fredholm, $B \in \calL\bigl(L_p(X,\mu)\bigr)$ and $K_1,K_2\in\calK\bigl(L_p(X,\mu)\bigr)$ such that $AB = I+K_1$ and $BA = I+K_2$. Fix $x\in X$, and for $j\in\N$ and $t>0$ let $\tilde{\varphi}_{j,t}:= \varphi_{j,t} - \varphi_{j,t}(x)$. Then
  \[[B,M_{\varphi_{j,t}}] = [B,M_{\tilde{\varphi}_{j,t}}] = B[M_{\tilde{\varphi}_{j,t}},A]B - BM_{\tilde{\varphi}_{j,t}} K_1 + K_2 M_{\tilde{\varphi}_{j,t}} B.\]
  Therefore,
  \begin{align} \label{eq:BDO_algebraic_c}
  &\sup\limits_{\norm{f}_p = 1} \Biggl(\sum\limits_{j = 1}^{\infty} \norm{[B,M_{\varphi_{j,t}}]f}_p^p\Biggr)^{1/p}\notag\\
  &\qquad \leq \sup\limits_{\norm{f}_p = 1} \left(\Biggl(\sum\limits_{j = 1}^{\infty} \norm{B[M_{\tilde{\varphi}_{j,t}},A]Bf}_p^p\Biggr)^{1/p} + \Biggl(\sum\limits_{j = 1}^{\infty} \norm{BM_{\tilde{\varphi}_{j,t}}K_1f}_p^p\Biggr)^{1/p} + \Biggl(\sum\limits_{j = 1}^{\infty} \norm{K_2M_{\tilde{\varphi}_{j,t}}Bf}_p^p\Biggr)^{1/p}\right)\notag\\
  &\qquad \leq \norm{B}^{2}\sup\limits_{\norm{g}_p = 1} \Biggl(\sum\limits_{j = 1}^{\infty} \norm{[M_{\tilde{\varphi}_{j,t}},A]g}_p^p\Biggr)^{1/P} + \norm{B}\sup\limits_{\norm{f}_p = 1} \Biggl(\sum\limits_{j = 1}^{\infty} \norm{M_{\tilde{\varphi}_{j,t}}K_1f}_p^p\Biggr)^{1/p}\\
  &\qquad \qquad + \norm{B}\sup\limits_{\norm{g}_p = 1} \Biggl(\sum\limits_{j = 1}^{\infty} \norm{K_2M_{\tilde{\varphi}_{j,t}}g}_p^p\Biggr)^{1/p},\notag
  \end{align}
  where we substituted $g := \frac{Bf}{\norm{Bf}_p}$ for $Bf \neq 0$. By Proposition \ref{prop:BDO_commutator}, the first term tends to $0$ as $t \to 0$.
  
  The net $(M_{\1_{B(x,R)}})_{R>0}$ converges $*$-strongly to the identity as $R \to \infty$. Since $K_1$ and $K_2$ are compact, Lemma \ref{lem:strong_and_compact_convergence} implies $\norm{M_{\1_{X\setminus B(x,R)}}K_1} \to 0$ as well as $\norm{K_2 M_{\1_{X\setminus B(x,R)}}}\to 0$ for $R \to \infty$.
  We have
  \begin{equation} \label{eq:BDO_algebraic_c2}
  \Biggl(\sum\limits_{j = 1}^{\infty} \norm{M_{\tilde{\varphi}_{j,t}}K_1f}_p^p\Biggr)^{1/p} \leq \Biggl(\sum\limits_{j = 1}^{\infty} \norm{M_{\tilde{\varphi}_{j,t}}M_{\1_{B(x,R)}}K_1f}_p^p\Biggr)^{1/p} + \Biggl(\sum\limits_{j = 1}^{\infty} \norm{M_{\tilde{\varphi}_{j,t}}M_{\1_{X\setminus B(x,R)}}K_1f}_p^p\Biggr)^{1/p}.
  \end{equation}
  For the second term we observe
  \begin{align*}
  &\Biggl(\sum\limits_{j = 1}^{\infty} \norm{M_{\tilde{\varphi}_{j,t}}M_{\1_{X\setminus B(x,R)}}K_1f}_p^p\Biggr)^{1/p}\\
  &\qquad \qquad \qquad \qquad \qquad \leq \Biggl(\sum\limits_{j = 1}^{\infty} \norm{M_{\varphi_{j,t}}M_{\1_{X\setminus B(x,R)}}K_1f}_p^p\Biggr)^{1/p} + \Biggl(\sum\limits_{j = 1}^{\infty} \varphi_{j,t}(x)^p\norm{M_{\1_{X\setminus B(x,R)}}K_1f}_p^p\Biggr)^{1/p}\\
  &\qquad \qquad \qquad \qquad \qquad = \norm{M_{\1_{X\setminus B(x,R)}}K_1f}_p + \norm{M_{\1_{X\setminus B(x,R)}}K_1f}_p\\
  &\qquad \qquad \qquad \qquad \qquad \leq 2\norm{M_{\1_{X\setminus B(x,R)}}K_1}\norm{f}_p,
  \end{align*}
  where we used $\sum\limits_{j = 1}^{\infty} \varphi_{j,t}^p = 1$. The second term in \eqref{eq:BDO_algebraic_c2} can therefore be made as small as desired by choosing $R$ large. For the first term in \eqref{eq:BDO_algebraic_c2} we assume $t \leq \frac{1}{R}$ as this implies
  \[\sum\limits_{j = 1}^{\infty} \abs{\tilde{\varphi}_{j,t}(y)}^p = \sum\limits_{j = 1}^{\infty} \abs{\varphi_{j,t}(y) - \varphi_{j,t}(x)}^p < t\]
  for all $y \in B(x,R)$ by Lemma \ref{lem:phis}. Hence
\[\sum\limits_{j = 1}^{\infty} \norm{M_{\tilde{\varphi}_{j,t}}M_{\1_{B(x,R)}}K_1f}_p^p < t\norm{K_1f}_p^p.\]  
  We thus infer that that the second term in \eqref{eq:BDO_algebraic_c} also tends to $0$ as $t \to 0$. Similarly, we may estimate the third term in \eqref{eq:BDO_algebraic_c}. Proposition \ref{prop:BDO_commutator} now yields the assertion.
  
  (d)
  In the proof of (c) we essentially showed that 
  \[\lim_{t\to 0} \sup\limits_{\norm{f}_p = 1} \sum\limits_{j = 1}^{\infty} \norm{M_{\tilde{\varphi}_{j,t}}Kf}_p^p  = \lim_{t\to 0} \sup\limits_{\norm{f}_p = 1} \sum\limits_{j = 1}^{\infty} \norm{KM_{\tilde{\varphi}_{j,t}}f}_p^p  = 0\]
  for all $K\in\calK\bigl(L_p(X,\mu)\bigr)$. Proposition \ref{prop:BDO_commutator} thus yields the assertion.
  
  (e)
  If $(A_n)_{n \in \N}$ is a sequence in $\BO$ with $\norm{A_n-A}\to 0$, then $(A_n^*)_{n \in \N}$ is a sequence in $\BO$ and we have $\norm{A_n^*-A^*}_{\calL\bigl(L_q(X,\mu)\bigr)} = \norm{A_n-A}_{\calL\bigl(L_p(X,\mu)\bigr)}\to 0$. Therefore $A^*\in \BDO^q$.
\end{proof}

In what follows we will work with operators which are not necessarily defined on the whole $L_p(X,\mu)$, but only acting on closed subspaces. A particular example are Toeplitz operators.

\begin{definition}
  Let $M^p\subseteq L_p(X,\mu)$ be a closed subspace and assume there exists a bounded linear projection $P\in \calL\bigl(L_p(X,\mu)\bigr)$ with $\ran P = M^p$.
  Let $Q:=I-P$. 
  \begin{enumerate}
    \item
      Then $A\in \calL(M^p)$ is called \emph{band-dominated}, provided $AP \in \BDO^p$. The set of band-dominated operators on $M^p$ will be denoted by $\calA^p$.
    \item
      Let $f\in L_\infty(X,\mu)$. Then the operator $T_f:= P M_f|_{M^p} \in \calL(M^p)$ is called \emph{Toeplitz-operator} associated with $f$.
      Let $\calT_{M^p}$ be the algebra generated by all Toeplitz operators on $M^p$.
  \end{enumerate}
\end{definition}

\begin{theorem}
\label{thm:Ap_properties}
  Let $(X,d)$ be a metric space that satisfies property \Aprime and $M^p\subseteq L_p(X,\mu)$ a closed subspace with projection $P\in \BDO^p$ onto $M^p$. Then $\calA^p$ is a closed subalgebra of $\calL(M^p)$, contains $\calK(M^p)$ , $\calT_{M^p}$ and is closed w.r.t.~Fredholm inverses. In particular, $\calA^p$ is inverse closed.
\end{theorem}

\begin{proof}
	This follows directly from Theorem \ref{thm:BDO_algebraic}.
\end{proof}

\section{Limit operators}
\label{sec:limit_operators}

In this section we turn to limit operators. We start by making three assumptions and then deduce properties of operators by means of their limit operators.

\begin{assumption}[Space]
\label{ass:space}
  Let $(X,d)$ be a proper metric space of bounded geometry (with $\varepsilon > 0$) that satisfies property \Aprime{}. Assume that $d$ is unbounded and let $\mu$ be a Radon measure on $X$. 
  Let $\beta X$ be the Stone-\v{C}ech compactification of $X$, $\Gamma X:=\beta X\setminus X$ the boundary.
\end{assumption}

\begin{assumption}[Subspaces and Projection]
\label{ass:subspaces_and_projection}
  Let $p\in (1,\infty)$ and let $M^p\subseteq L_p(X,\mu)$ be a closed subspace
  with bounded projection $P \in \BDO^p$.
  Moreover, assume $M_{\1_K}P, PM_{\1_K}\in\calK\bigl(L_p(X,\mu)\bigr)$ for all compact subsets $K\subseteq X$.
\end{assumption}

\begin{assumption}[Shifts]
\label{ass:shifts}
  Fix $x_0\in X$.
  For $x\in X$ let $\phi_x\from X\to X$ be a bijective isometry with $\phi_x(x_0) = x$.
  Assume that $\mu\circ\phi_x \ll \mu \ll \mu\circ \phi_x$ and let $h_x$ be a measurable function such that $\abs{h_x}^p =\frac{\mathrm{d}(\mu\circ\phi_x)}{\mathrm{d}\mu}$ $\mu$-almost everywhere.
  Assume that the maps $x \mapsto \phi_x(y)$ and $x \mapsto h_x(y)$ are continuous for $\mu$-almost every $y \in X$.
  For $p\in (1,\infty)$ and $x\in X$ let $U_x^p\from L_p(X,\mu)\to L_p(X,\mu)$ be defined by $U_x^p f := (f \circ \phi_x) \cdot h_x$ and assume that $x \mapsto M_{\1_K}U_x^pP(U_x^p)^{-1}M_{\1_{K'}}$ extends continuously to $\beta X$ for all $K,K' \subseteq X$ compact.
\end{assumption}

These assumptions are used throughout the rest of this paper. Here are some additional remarks concerning these assumptions:

\begin{remark}
\begin{itemize}
	\item[(a)] As we assume that $d$ is unbounded, $(X,d)$ cannot be compact. In fact, as $X$ is proper, assuming that $(X,d)$ is not compact is equivalent to assuming that $d$ is unbounded.
	\item[(b)] The Radon property of $\mu$ guarantees that the continuous functions with compact support are dense in $L^p(X,\mu)$.
	\item[(c)] The inverses $\phi_x^{-1}$ are also isometries.
	\item[(d)] If $(X,d)$ is proper and satisfies the symmetry assumptions in Assumption  \ref{ass:shifts}, then $(X,d)$ automatically has bounded geometry with arbitrary $\varepsilon > 0$. Indeed, for every $r > 0$ the closed ball $B[x_0,r]$ is compact and hence can be convered with finitely many balls of radius $\varepsilon$. As isometries map balls to balls, the bounded geometry follows.
\end{itemize}
\end{remark}

We proceed with a few simple lemmas which follow directly from the assumptions. For the reader's convenience we also provide short proofs. In the following, we use the notation $L_{\infty,\rm c}(X,\mu)$ for the essentially bounded functions with compact support.

\begin{lemma}
\label{lem:distance_under_phi_x}
  Let $f,g\in L_{\infty,\rm c}(X,\mu)$ and $x\in X$. Then
  \[\dist\bigl(\spt (f\circ\phi_x^{-1}), \spt(f\circ \phi_x^{-1})\bigr) = \dist(\spt f, \spt g).\]
\end{lemma}

\begin{proof}
  We have
  \begin{align*}
    \dist\bigl(\spt (f\circ\phi_x^{-1}), \spt(f\circ \phi_x^{-1})\bigr) & = \inf_{y\in \spt (f\circ\phi_x^{-1})} \inf_{z\in \spt(g\circ \phi_x^{-1})} d(y,z) \\
    & = \inf_{y\in \spt (f\circ\phi_x^{-1})} \inf_{z\in \spt(g\circ \phi_x^{-1})} d(\phi_x(y),\phi_x(z)) \\
    & = \inf_{y'\in \spt f} \inf_{z'\in \spt g} d(y',z')\\
    & = \dist(\spt f, \spt g). \qedhere
  \end{align*}
\end{proof}

\begin{lemma}
\label{lem:surjective_isometry}
  Let $p\in (1,\infty)$, $x\in X$. Then $U_x^p$ is a surjective isometry. 
\end{lemma}

\begin{proof}
  Let $f\in L_p(X,\mu)$. Then
  \[\int_X \abs{U_x^p f(y)}^p \, \mathrm{d}\mu(y) = \int_{X} \abs{f(\phi_x(y))}^p \abs{h_x(y)}^p \, \mathrm{d}\mu(y) = \int_{\phi_x^{-1}(X)} \abs{f(\phi_x(y))}^p \,\mathrm{d}(\mu\circ \phi_x)(y) = \int_X \abs{f(y)}^p \, \mathrm{d}\mu(y).\]
  Hence, $U_x^p$ is an isometry.
  
  Let $g\in L_p(X,\mu)$. Define $f:=\Bigl(\frac{g}{h_x}\Bigr)\circ \phi_x^{-1}$. Then $f\in L_p(X,\mu)$ and $U_x^pf = g$, so $U_x^p$ is surjective.
\end{proof}

Note that $(U_x^p)^{-1}f = \Bigl(\frac{f}{h_x}\Bigr)\circ \phi_x^{-1}$ for all $f\in L_p(X,\mu)$, $x\in X$, and $(U_x^p)^{-1}$ is again a surjective isometry for all $x\in X$.

\begin{lemma}
\label{lem:multiplication_and_shifts}
  Let $f\in L_\infty(X,\mu)$, $x\in X$. Then $M_f U_{x}^p = U_x^p M_{f\circ \phi_x^{-1}}$.
\end{lemma}

\begin{proof}
  Let $g\in L_p(X,\mu)$. Then for $\mu$-almost every $y\in X$ we have
  \begin{align*}
    M_f U_{x}^p g (y) & = f(y) g(\phi_x(y)) h_x(y) = (f\circ\phi_x^{-1})(\phi_x(y)) g(\phi_x(y)) h_x(y) = U_x^p M_{f\circ \phi_x^{-1}}g(y).\qedhere
  \end{align*}
\end{proof}

\begin{proposition}
\label{prop:weak_limit_operators}
Let $B \in \calL(L_p(X,\mu))$. Then the map $x \mapsto U_x^pB(U_x^p)^{-1}$ has a weakly continuous extension to $\beta X$.
\end{proposition}

\begin{proof}
  We first show that the map $X\ni x \mapsto U_x^pB(U_x^p)^{-1}$ is strongly continuous for any $B \in \calL(L_p(X,\mu))$. Let $f \in L_p(X,\mu)$ be continuous. 
  As $x \mapsto \phi_x(y)$ and $x \mapsto h_x(y)$ are continuous for $\mu$-almost every $y \in X$ by Assumption \ref{ass:shifts}, 
  $x \mapsto (U_x^pf)(y) = f(\phi_x(y))h_x(y)$ is also continuous for $\mu$-almost every $y\in X$.
  Moreover, $\norm{U_x^pf}_p = \norm{f}_p$ for all $x \in X$, hence $x \mapsto U_x^pf$ is continuous by Scheff\'e's Lemma. 
  As this is true for every continuous function $f$ and the continuous $L_p$-functions are dense in $L_p(X,\mu)$,
  we obtain that $x \mapsto U_x^p$ is strongly continuous. This also implies that $x \mapsto (U_x^p)^{-1}$ is strongly continuous. Indeed, if $(x_\iota)_\iota$ is a net in $X$ with $x_\iota \to x \in X$ and $f \in L_p(X,\mu)$, then
  \[\norm{\bigl((U_{x_\iota}^p)^{-1}-(U_x^p)^{-1}\bigr)f}_p = \norm{(U_{x_\iota}^p)^{-1}\bigl(U_x^p - U_{x_\iota}^p\bigr)(U_x^p)^{-1}f}_p = \norm{\bigl(U_x^p - U_{x_\iota}^p\bigr)(U_x^p)^{-1}f}_p \to 0\]
  as $x_\iota \to x$. As a consequence, $x \mapsto U_x^pB(U_x^p)^{-1}$ is also strongly continuous. As bounded sets are relatively compact in the weak operator topology, the map $x \mapsto U_x^pB(U_x^p)^{-1}$ has a weakly continuous extension to $\beta X$.
\end{proof}

Proposition \ref{prop:weak_limit_operators} shows that for $B \in \calL(L_p(X,\mu))$, $x \in \beta X$ and every net $(x_\iota)_\iota$ in $\beta X$ with $x_\iota \to x$ the weak limit
  \begin{equation} \label{eq:weak_limit}
  B_x := \wlim\limits_{x_\iota \to x} U_{x_\iota}^pB(U_{x_\iota}^p)^{-1}
  \end{equation}
  exists and is independent of the net $(x_{\iota})$. In particular, $B_x = U_x^pB(U_x^p)^{-1}$ for $x \in X$.

\begin{corollary}
\label{cor:limit_operator_norm}
Let $B \in \calL(L_p(X,\mu))$. Then $\norm{B_x} \leq \norm{B}$ for all $x \in \beta X$.
\end{corollary}

\begin{proof}
This follows directly from the uniform boundedness principle and the fact that $U_{x_\iota}^p$ and $(U_{x_\iota}^p)^{-1}$ are isometries.
\end{proof}

\begin{definition}[limit operators]
\label{def:limit_operator}
  Let $A\in\calL(M^p)$, $x \in \bdry{X}$ and $(x_{\iota})$ a net in $X$ that converges to $x$. Then the operator
  \[A_x := \wlim\limits_{x_\iota\to x} U_{x_\iota}^p AP(U_{x_\iota}^p)^{-1}|_{\ran(P_x)}\]
  is well-defined and called a \emph{limit operator} of $A$.
\end{definition}

Note that $A_x$ is by definition an operator on the closed subspace $\ran(P_x) \subseteq L_p(X,\mu)$, which may depend on $x$ and may differ from $M^p$. For band-dominated operators the convergence is much stronger:

\begin{theorem} \label{thm:limit_operator_existence}
Let $A\in\calA^p$, $K \subseteq X$ compact and $x \in \beta X$. For every net $(x_{\iota})$ in $\beta X$ converging to $x$ we have
\begin{itemize}
	\item[(i)] $\lim\limits_{x_\iota\to x} M_{\1_K}(AP)_{x_{\iota}} = M_{\1_K}A_xP_x$,
	\item[(ii)] $\lim\limits_{x_\iota\to x} (AP)_{x_\iota}M_{\1_K} = A_xP_xM_{\1_K}$,
	\item[(iii)] $\slim\limits_{x_\iota\to x} (AP)_{x_{\iota}} = A_xP_x$,
	\item[(iv)] $\slim\limits_{x_\iota\to x} \bigl((AP)_{x_\iota}\bigr)^* = P_x^*A_x^*$.
\end{itemize}
\end{theorem}

For better readability we will divide the proof into several propositions and lemmas, which are useful and interesting in their own right.

\begin{proposition}
\label{prop:BO_limit_ops}
If $B \in \BO$, then $B_x \in \BO$ and $\prop(B_x) \leq \prop(B)$ for all $x \in \beta X$.
\end{proposition}

\begin{proof}
  Let $(x_\iota)$ be a net in $X$ that converges to $x$ and let $K,K'\subseteq X$ be compact with $\dist(K,K')>\prop(B)$. Then Lemma \ref{lem:distance_under_phi_x} yields
  \[\dist\bigl(\spt (\1_K\circ\phi_x^{-1}), \spt(\1_{K'}\circ \phi_x^{-1})\bigr)>\prop(B)\]
  and by Lemma \ref{lem:multiplication_and_shifts} we obtain
  \[M_{\1_K} U_{x_\iota}^pB(U_{x_\iota}^p)^{-1} M_{\1_{K'}} = U_{x_\iota}^p M_{\1_K\circ\phi_{x_\iota}^{-1}}BM_{\1_{K'}\circ \phi_{x\iota}^{-1}} (U_{x_\iota}^p)^{-1} = 0\]
  for all $\iota$. Taking the weak limit $x_{\iota} \to x$, we obtain $B_x \in \BO$ and $\prop(B_x) \leq \prop(B)$.
\end{proof}

\begin{corollary}
\label{cor:BDO_limit_ops}
If $B \in \BDO^p$, then $B_x \in \BDO^p$ for all $x \in \beta X$.
\end{corollary}

\begin{proof}
  For $B \in \BDO^p$ and $\varepsilon > 0$ there is a $B_0 \in \BO$ such that $\|B-B_0\| < \varepsilon$. Corollary \ref{cor:limit_operator_norm} yields
  \[\|B_x - (B_0)_x\| = \|(B-B_0)_x\| \leq \|B-B_0\| < \varepsilon,\]
  which implies $B_x \in \BDO^p$ by Theorem \ref{thm:BDO_algebraic} and Proposition \ref{prop:BO_limit_ops}.
\end{proof}

\begin{lemma}
\label{lem:BDO_commutator_uniform}
Let $B \in \BDO^p$. For all $\varepsilon > 0$ there is a $t_0$ such that for all $t \leq t_0$ and all $x \in \beta X$ we have
  \[\sup\limits_{\norm{f}_p = 1} \sum\limits_{j = 1}^{\infty} \norm{[B_x,M_{\varphi_{j,t}}]f}_p^p \leq \varepsilon.\]
\end{lemma}

\begin{proof}
  For band operators this is Lemma \ref{lem:BO_commutator} combined with Proposition \ref{prop:BO_limit_ops}. For band-dominated operators we can use approximation as in Proposition \ref{prop:BDO_commutator} and Corollary \ref{cor:BDO_limit_ops}.
\end{proof}

\begin{lemma}
\label{lem:compact_convergence}
Let $B \in \BDO^p$ and let $(x_{\iota})$ be a net in $\beta X$ that converges to $x \in \beta X$. Assume
\[\lim\limits_{x_{\iota} \to x} \norm{M_{\1_K}\left(B_{x_\iota} - B_x\right)M_{\1_{K'}}} = 0.\]
for all $K,K' \subseteq X$ compact. Then also
\[\lim\limits_{x_{\iota} \to x} \norm{M_{\1_K}\left(B_{x_\iota} - B_x\right)} = 0 \quad \text{and} \quad \lim\limits_{x_{\iota} \to x} \norm{\left(B_{x_\iota} - B_x\right)M_{\1_{K'}}} = 0\]
for all $K,K' \subseteq X$ compact.
\end{lemma}

\begin{proof}
(i) Let $\varepsilon > 0$ and $K \subseteq X$ compact. By Lemma \ref{lem:BDO_commutator_uniform} we can choose a $t>0$ such that for all $y \in \beta X$ we have
\begin{equation} \label{eq:compact_convergence}
  \sup\limits_{\norm{f}_p = 1} \sum\limits_{j = 1}^{\infty} \norm{[B_y,M_{\varphi_{j,t}}]f}_p^p \leq \varepsilon
\end{equation}
  Choose $j_0 \in \N$ sufficiently large such that $M_{\varphi_{j,t}^p}M_{\1_K} = 0$ for $j \geq j_0$. Then, using Lemma \ref{lem:sum_estimate} and Corollary \ref{cor:limit_operator_norm}, we get
  \begin{align} \label{eq:limit_operator_existence2}
  \norm{M_{\1_K}\left(B_{x_{\iota}} - B_x\right)} &= \Biggl\|\sum\limits_{j < j_0} M_{\varphi_{j,t}^p}M_{\1_K}\left(B_{x_{\iota}} - B_x\right)\Biggr\|\notag\\
  &= \sup\limits_{\norm{f}_p = 1} \Biggl\|\sum\limits_{j < j_0} M_{\varphi_{j,t}^p}M_{\1_K}\left(B_{x_{\iota}} - B_x\right)f\Biggr\|_p\notag\\
  &\leq \sup\limits_{\norm{f}_p = 1} \Biggl(\sum\limits_{j < j_0} \norm{M_{\varphi_{j,t}}M_{\1_K}\left(B_{x_{\iota}} - B_x\right)f}_p^p\Biggr)^{1/p}\notag\\
  &\leq \sup\limits_{\norm{f}_p = 1} \Biggl(\sum\limits_{j < j_0} \norm{M_{\1_K}\left[M_{\varphi_{j,t}},B_{x_{\iota}} - B_x\right]f}_p^p\Biggr)^{1/p}\notag\\
  &\quad + \sup\limits_{\norm{f}_p = 1} \Biggl(\sum\limits_{j < j_0} \norm{M_{\1_K}\left(B_{x_{\iota}} - B_x\right)M_{\varphi_{j,t}}f}_p^p\Biggr)^{1/p}\notag\\
  &\leq 2\norm{B}\varepsilon^{1/p} + \sup\limits_{\norm{f}_p = 1} \Biggl(\sum\limits_{j < j_0} \norm{M_{\1_K}\left(B_{x_{\iota}} - B_x\right)M_{\varphi_{j,t}}f}_p^p\Biggr)^{1/p}.
  \end{align}
  Let $K' \subseteq X$ be a compact set such that $M_{\varphi_{j,t}} = M_{\1_{K'}}M_{\varphi_{j,t}}$ for all $j < j_0$. It follows
  \begin{align*}
  \sup\limits_{\norm{f}_p = 1} \Biggl(\sum\limits_{j < j_0} \norm{M_{\1_K}\left(B_{x_{\iota}} - B_x\right)M_{\varphi_{j,t}}f}_p^p\Biggr)^{1/p} &= \sup\limits_{\norm{f}_p = 1} \Biggl(\sum\limits_{j < j_0} \norm{M_{\1_K}\left(B_{x_{\iota}} - B_x\right)M_{\1_{K'}}M_{\varphi_{j,t}}f}_p^p\Biggr)^{1/p}\\
  &\leq \norm{M_{\1_K}\left(B_{x_{\iota}} - B_x\right)M_{\1_{K'}}}\sup\limits_{\norm{f}_p = 1} \Biggl(\sum\limits_{j < j_0} \norm{M_{\varphi_{j,t}}f}_p^p\Biggr)^{1/p}\\
  &\leq \norm{M_{\1_K}\left(B_{x_{\iota}} - B_x\right)M_{\1_{K'}}},
  \end{align*}
	which tends to $0$ by assumption. As $\varepsilon$ was arbitrary, we conclude
  \[\lim\limits_{x_{\iota} \to x} \norm{M_{\1_K}\left(B_{x_\iota} - B_x\right)} = 0.\]
	
	(ii) Let $\varepsilon > 0$ and $K' \subseteq X$ compact. Again, we choose $t>0$ such that \eqref{eq:compact_convergence} holds and $j_0 \in \N$ sufficiently large such that $M_{\varphi_{j,t}}M_{\1_{K'}} = 0$ for $j \geq j_0$. We have
  \[\norm{\left(B_{x_\iota} - B_x\right)M_{\1_{K'}}} \leq \Biggl\|\sum\limits_{j < j_0} M_{\varphi_{j,t}^p}\left(B_{x_\iota} - B_x\right)M_{\1_{K'}}\Biggr\| + \Biggl\|\sum\limits_{j \geq j_0} M_{\varphi_{j,t}^p}\left(B_{x_\iota} - B_x\right)M_{\1_{K'}}\Biggr\|.\]
  The first term tends to $0$ as $\sum\limits_{j < j_0} \varphi_{j,t}^p$ has compact support. The second term can be estimated using Lemma \ref{lem:sum_estimate}:
  \begin{align*}
  \Biggl\|\sum\limits_{j \geq j_0} M_{\varphi_{j,t}^p}\left(B_{x_\iota} - B_x\right)M_{\1_{K'}}\Biggr\| &= \sup\limits_{\norm{f}_p = 1} \Biggl\|\sum\limits_{j \geq j_0} M_{\varphi_{j,t}^p}\left(B_{x_\iota} - B_x\right)M_{\1_{K'}}f\Biggr\|_p\\
  &\leq \sup\limits_{\norm{f}_p = 1}\Biggl(\sum\limits_{j \geq j_0} \norm{M_{\varphi_{j,t}}\left(B_{x_\iota} - B_x\right)M_{\1_{K'}}f}_p^p\Biggr)^{1/p}\\
  &= \sup\limits_{\norm{f}_p = 1}\Biggl(\sum\limits_{j \geq j_0} \norm{\left[M_{\varphi_{j,t}},B_{x_\iota} - B_x\right]M_{\1_{K'}}f}_p^p\Biggr)^{1/p}\\
  &\leq 2\varepsilon^{1/p}.
  \end{align*}
  As $\varepsilon$ was arbitrary, we again conclude
  \[\lim\limits_{x_{\iota} \to x} \norm{\left(B_{x_\iota} - B_x\right)M_{\1_{K'}}} = 0.\qedhere\]
\end{proof}

\begin{corollary}
\label{cor:strong_convergence}
Let $B \in \BDO^p$ and let $(x_{\iota})$ be a net in $\beta X$ that converges to $x \in \beta X$. Assume
\[\lim\limits_{x_{\iota} \to x} \norm{M_{\1_K}\left(B_{x_\iota} - B_x\right)M_{\1_{K'}}} = 0.\]
for all $K,K' \subseteq X$ compact. Then $B_{x_\iota} \to B_x$ in the strong operator topology.
\end{corollary}

\begin{proof}
Let $\varepsilon > 0$ and $f \in L_p(X,\mu)$. Choose $K \subseteq X$ compact such that $\norm{M_{\1_{K^c}}f}_p < \varepsilon$. Then
\begin{align*}
\norm{\left(B_{x_\iota} - B_x\right)f}_p &\leq \norm{\left(B_{x_\iota} - B_x\right)M_{\1_K}f}_p + \norm{\left(B_{x_\iota} - B_x\right)M_{\1_{K^c}}f}_p\\
&\leq \norm{\left(B_{x_\iota} - B_x\right)M_{\1_K}}\norm{f}_p + 2\norm{B}\varepsilon
\end{align*}
by Corollary \ref{cor:limit_operator_norm}. As $\varepsilon$ was arbitrary, Lemma \ref{lem:compact_convergence} yields the assertion.
\end{proof}

\begin{lemma} \label{lem:range_of_limops}
\begin{itemize}
	\item[(i)] $P_x$ is a projection for every $x \in \beta X$.
	\item[(ii)] $M_{\1_K}P_x$ and $P_xM_{\1_K}$ are compact for all compact sets $K \subseteq X$ and $x \in \beta X$.
	\item[(iii)] Let $B := AP$ for $A \in \calA^p$. Then $B_x = B_xP_x = P_xB_x$ for every $x \in \beta X$.
\end{itemize}
\end{lemma}

\begin{proof}
First assume that $x \in X$. Then $P_x = U_x^pP(U_x^p)^{-1}$ is a projection onto $\ran(U_x^pP)$ with $\|P_x\| = \|P\|$. Also,
\[M_{\1_K}P_x = M_{\1_K}U_x^pP(U_x^p)^{-1} = U_x^pM_{\1_{\phi_x(K)}}P(U_x^p)^{-1}\]
(see Lemma \ref{lem:multiplication_and_shifts}) and the latter is compact by Assumption \ref{ass:subspaces_and_projection}. Similarly, $P_xM_{\1_K}$ is compact. Moreover,
\[B_x = U_x^pB(U_x^p)^{-1} = U_x^pPB(U_x^p)^{-1} = U_x^pP(U_x^p)^{-1}U_x^pB(U_x^p)^{-1} = P_xB_x\]
and similarly $B_x = B_xP_x$.

Now let $x \in \Gamma X$. Then there is a net $(x_{\iota})$ in $X$ with $x_{\iota} \to x$. By Assumption \ref{ass:shifts} and Corollary \ref{cor:strong_convergence}, we have $P_{x_{\iota}} \to P_x$ strongly. This implies that $P_x$ is again a projection. Also, Lemma \ref{lem:compact_convergence} implies $M_{\1_K}P_{x_{\iota}} \to M_{\1_K}P_x$ and $P_{x_{\iota}}M_{\1_K} \to P_xM_{\1_K}$ so that $M_{\1_K}P_x$ and $P_xM_{\1_K}$ are again compact. Moreover, by Lemma \ref{lem:compact_convergence}, we have
	\begin{equation} \label{eq:lem_range_of_limops}
	\norm{M_{\1_K}(P_{x_{\iota}} - P_x)} \to 0 \quad \text{and} \quad \norm{(P_{x_{\iota}} - P_x)M_{\1_{K'}}} \to 0
	\end{equation}
	as $x_{\iota} \to x$. Let $\varepsilon > 0$, $f \in L_p(X,\mu)$, $g \in L_q(X,\mu)$ and choose $K \subseteq X$ compact such that $\norm{M_{\1_{K^c}}g}_q < \varepsilon$. With the usual dual pairing $\dupa{\cdot}{\cdot}$ of $L_p(X,\mu)$ and $L_q(X,\mu)$ we get
  \begin{align*}
  \dupa{\left(B_{x_{\iota}} - P_xB_x\right)f}{g} &= \dupa{\left(P_{x_{\iota}}B_{x_{\iota}} - P_xB_x\right)f}{g}\\
  &= \dupa{\left(P_{x_{\iota}} - P_x\right)B_{x_{\iota}}f}{g} + \dupa{P_x\left(B_{x_{\iota}} - B_x\right)f}{g}\\
  &= \dupa{M_{\1_K}\left(P_{x_{\iota}} - P_x\right)B_{x_{\iota}}f}{g} + \dupa{\left(P_{x_{\iota}} - P_x\right)B_{x_{\iota}}f}{M_{\1_{K^c}}g} + \dupa{\left(B_{x_{\iota}} - B_x\right)f}{P_x^*g}.
  \end{align*}
  The first and the third term tend to $0$ as $x_{\iota} \to x$ by \eqref{eq:lem_range_of_limops} and Proposition \ref{prop:weak_limit_operators}, whereas the second term is bounded by $2\norm{P}\norm{B}\norm{f}_p\varepsilon$. As $\varepsilon$ was arbitrary, we get $B_{x_\iota} \to P_xB_x$ weakly, which implies $P_xB_x = B_x$. Similarly, we obtain $B_xP_x = P_x$.
\end{proof}

\begin{proof}[Proof of Theorem \ref{thm:limit_operator_existence}]
  Let $A \in \calA^p$, $K,K' \subseteq X$ compact and let $(x_{\iota})$ be a net in $X$ converging to $x \in \beta X$. Define $B := AP \in \BDO^p$. By Lemma \ref{lem:compact_convergence}, Corollary \ref{cor:strong_convergence} and Lemma \ref{lem:range_of_limops}(iii), it suffices to show
	\[\lim\limits_{x_{\iota} \to x} \norm{M_{\1_K}\left(B_{x_\iota} - B_x\right)M_{\1_{K'}}} = 0.\]
	By Proposition \ref{prop:weak_limit_operators}, the weak limits
  \[P_x := \wlim\limits_{x_\iota \to x} P_{x_\iota} \quad \text{and} \quad B_x := \wlim\limits_{x_\iota \to x} B_{x_\iota}\]
  exist and do not depend on the net converging to $x$. By Corollary \ref{cor:limit_operator_norm} and Lemma \ref{lem:range_of_limops},
  \begin{align*}
  \norm{M_{\1_K}(B_{x_{\iota}} - B_x)M_{\1_{K'}}} &= \norm{M_{\1_K}(P_{x_{\iota}}B_{x_{\iota}}P_{x_{\iota}} - P_xB_xP_x)M_{\1_{K'}}}\\
  	&\leq \norm{M_{\1_K}(P_{x_{\iota}} - P_x)B_{x_{\iota}}P_{x_{\iota}}M_{\1_{K'}}} + \norm{M_{\1_K}P_xB_{x_{\iota}}(P_{x_{\iota}} - P_x)M_{\1_{K'}}}\\
	&\quad + \norm{M_{\1_K}P_x(B_{x_{\iota}} - B_x)P_xM_{\1_{K'}}}\\
	&\leq \norm{M_{\1_K}(P_{x_{\iota}} - P_x)}\norm{B}\norm{P} + \norm{P}\norm{B}\norm{(P_{x_{\iota}} - P_x)M_{\1_{K'}}}\\
	&\quad + \norm{M_{\1_K}P_x(B_{x_{\iota}} - B_x)P_xM_{\1_{K'}}}.
  \end{align*}
  The first and the second term tend to $0$ by Lemma \ref{lem:compact_convergence}. $M_{\1_K}P_x$ and $P_xM_{\1_{K'}}$ are compact by Lemma \ref{lem:range_of_limops}. Therefore and since $B_{x_{\iota}} \to B_x$ weakly, the third term also tends to $0$ as $x_{\iota} \to x$.
\end{proof}

The above may of course also be applied to $Q := I-P$. The next corollary is therefore immediate from Assumption \ref{ass:shifts}, Theorem \ref{thm:limit_operator_existence}, Lemma \ref{lem:compact_convergence} and Corollary \ref{cor:strong_convergence}.

\begin{corollary} \label{cor:hats}
Let $A \in \calA^p$, $K \subseteq X$ compact and $(x_{\iota})$ a net in $\beta X$ converging to some $x \in \beta X$. Define $\hat{A} := AP+Q$. Then $\hat{A}_y = A_yP_y+Q_y$ for all $y \in \beta X$ and
\begin{itemize}
	\item[(i)] $\lim\limits_{x_\iota\to x} M_{\1_K}\hat{A}_{x_{\iota}} = M_{\1_K}\hat{A}_x$,
	\item[(ii)] $\lim\limits_{x_\iota\to x} \hat{A}_{x_{\iota}}M_{\1_K} = \hat{A}_xM_{\1_K}$,
	\item[(iii)] $\slim\limits_{x_\iota\to x} \hat{A}_{x_{\iota}}= \hat{A}_x$,
	\item[(iv)] $\slim\limits_{x_\iota\to x} \hat{A}_{x_{\iota}}^* = \hat{A}_x^*$.
\end{itemize}
\end{corollary}

In the next proposition we summarize a few properties of limit operators, which follow directly from properties of the strong operator convergence.

\begin{proposition}
\label{prop:limit_operator_properties}
  Let $A,B\in \calA^p$, $(A_n)_{n \in \N}$ a sequence in $\calA^p$ and $x\in \beta X$. Then
  \begin{itemize}
  	\item[(a)] $A_xP_x \in \BDO^p$,
  	\item[(b)] $(A+B)_x = A_x + B_x$,
  	\item[(c)] $(AB)_x = A_xB_x$,
  	\item[(d)] $\norm{A_x} \leq \norm{A}$,
  	\item[(e)] if $A_n \to A$ in norm, then $(A_n)_x \to A_x$ in norm.
  \end{itemize}
\end{proposition}

\begin{proof}
  (a) follows from Corollary \ref{cor:BDO_limit_ops}. (b), (c) and (e) follow from standard properties of strong convergence and the fact that $\set{U_y^pAP(U_y^p)^{-1} :\; y \in X}$ is a bounded set. For (d) let $f \in \ran(P_x)$ and $(x_{\iota})$ a net in $X$ converging to $x$. Using that $U_{x_\iota}^p$ is an isometry, we get
  \begin{align*}
  \norm{A_xf}_p &\leq \norm{U_{x_\iota}^p AP(U_{x_\iota}^p)^{-1}f}_p + \norm{\left(A_x - U_{x_\iota}^p AP(U_{x_\iota}^p)^{-1}\right)f}_p\\
  &\leq \norm{A}\norm{U_{x_\iota}^p P(U_{x_\iota}^p)^{-1}f}_p + \norm{\bigl(A_x - U_{x_\iota}^p AP(U_{x_\iota}^p)^{-1}\bigr)f}_p\\
  &\leq \norm{A}\norm{\bigl(U_{x_\iota}^p P(U_{x_\iota}^p)^{-1} - P_x\bigr)f}_p + \norm{A}\norm{f}_p + \norm{\bigl(A_x - U_{x_\iota}^p AP(U_{x_\iota}^p)^{-1}\bigr)f}_p
  \end{align*}
  Taking the limit $x_{\iota} \to x$, we get $\norm{A_xf}_p \leq \norm{A}\norm{f}_p$.
\end{proof}

\subsection{Compact Operators}

In this subsection we finally show (I). One implication is straightforward to prove.

\begin{proposition}
\label{prop:limit_operators_of_compact_operators}
  Let $K\in\calK(L_p(X,\mu))$. Then $K_x = 0$ for all $x\in \Gamma X$. In particular, $K_x = 0$ for all $x\in \Gamma X$ and $K\in\calK(M^p)$. 
\end{proposition}

\begin{proof}
  Let $x_0 \in X$, $(x_\iota)$ a net in $X$ that converges to $x \in \Gamma X$ and $R > 0$. Then
  \[\norm{U_{x_\iota}^p K(U_{x_\iota}^p)^{-1}M_{\1_{B[x_0,R]}}} = \norm{KM_{\1_{B[x_{\iota},R]}}}\]
  by Lemma \ref{lem:multiplication_and_shifts}. %Fix $\varepsilon > 0$, $g \in L_p(X,\mu)$ and choose $r > 0$ sufficiently large such that $\norm{\1_{X \setminus B(x_0,r)} \cdot g}_p < \varepsilon$. As $X$ is assumed to be proper, $B[x_0,r+R] \subset X$ is compact. Therefore $\beta X \setminus B[x_0,r+R]$ is an open neighborhood of $x \in \Gamma X$ and there exists a $\iota_0$ such that for all $\iota \geq \iota_0$ we have $x_{\iota} \in X \setminus B[x_0,r+R]$. It follows
%  \[\norm{M_{\1_{B[x_\iota,R]}}g}_p \leq \norm{M_{\1_{X \setminus B(x_0,r)}}g}_p < \varepsilon.\]
%  As $g$ and $\varepsilon$ were arbitrary, 
$\bigl(M_{\1_{B[x_\iota,R]}}\bigr)_\iota$ converges $*$-strongly to $0$ as $x_{\iota} \to x$. Compactness of $K$ therefore implies
  \[\norm{KM_{\1_{B[x_\iota,R]}}} \to 0\]
  as $x_{\iota} \to x$. Hence 
	\[\lim\limits_{x_{\iota} \to x} \norm{U_{x_\iota}^p K(U_{x_\iota}^p)^{-1}M_{\1_{K'}}} = 0\]
	for all $K' \subseteq X$ compact. Thus, $K_x = 0$.
\end{proof}

The reverse implication is more difficult to show and requires the following notions. For $t>0$ let
$r_t:=\sup\limits_{j\in\N} \diam \spt \varphi_{j,t}$, which is finite by Lemma \ref{lem:phis} and Assumption \ref{ass:space}.

\begin{definition}
  For $A\in \calL\bigl(L_p(X,\mu)\bigr)$, $F\subseteq X$ a Borel set and $t>0$ we define
  \begin{align*}
    \tripnorm{A|_F}_t & := \sup\set{\norm{Af}_p:\; f\in L_p(X,\mu),\, \norm{f}_p=1,\, \spt f\subseteq B[x,r_t]\cap F\,\text{for some $x\in X$}},\\
    \norm{A|_F} & := \sup\set{\norm{Af}_p:\; f\in L_p(X,\mu),\, \norm{f}_p=1,\, \spt f\subseteq  F}.
  \end{align*}
\end{definition}

These definitions are similar to \cite[Definition 3.3]{HaggerLindnerSeidel2016} in case of $\ell_p$-spaces over $\Z^N$.
They yields localized norms of the operator. The following proposition is then a version of \cite[Corollary 3.5]{HaggerLindnerSeidel2016} to our case of $L_p$-spaces over metric measure spaces of bounded geometry.

\begin{proposition}
\label{prop:local_norms}
  Let $A\in \BDO^p$ and $\varepsilon>0$. Then there is a $t_0>0$ such that for all $t \leq t_0$, all Borel sets $F\subseteq X$ and all operators $B\in\set{A}\cup\set{A_x:\; x\in \beta X}$ we have
  \[\norm{B|_F} \geq \tripnorm{B|_F}_t \geq \norm{B|_F} - \varepsilon.\]
\end{proposition}

\begin{proof}
  Let $\varepsilon > 0$, $B\in \set{A}\cup\set{A_x:\; x\in \beta X}$ and $F\subseteq X$ a Borel set. Let $f\in L_p(X,\mu)$ with $\norm{f}_p = 1$ and $\spt f\subseteq F$, such that 
  \[\norm{Bf}_p \geq \norm{B|_F} - \frac{1}{2}\varepsilon.\]
  By Lemma \ref{lem:BDO_commutator_uniform} there is a $t_0>0$ (independent of $B$ and $f$) such that for all $t \leq t_0$:
  \[\Biggl(\sum_{j=1}^\infty \norm{[B, M_{\varphi_{j,t}}]f}_p^p\Biggr)^{1/p} \leq \frac{1}{2}\varepsilon.\]
  Since $\sum_{j=1}^\infty \varphi_{j,t}(x)^p = 1$ for all $x\in X$, we have
  \[\Biggl(\sum_{j=1}^\infty \norm{M_{\varphi_j,t} Bf}_p^p\Biggr)^{1/p} = \norm{Bf}_p \geq \norm{B|_F} - \frac{1}{2}\varepsilon.\]
  Minkowski's inequality yields
  \begin{align*}
    \Biggl(\sum_{j=1}^\infty \norm{BM_{\varphi_j,t} f}_p^p\Biggr)^{1/p} 
    & \geq \Biggl(\sum_{j=1}^\infty \norm{M_{\varphi_j,t} Bf}_p^p\Biggr)^{1/p} - \Biggl(\sum_{j=1}^\infty \norm{[B,M_{\varphi_{j,t}}]f}_p^p\Biggr)^{1/p} \\
    & \geq \norm{B|_F} - \varepsilon\\
    &= \bigl(\norm{B|_F} - \varepsilon\bigr) \Biggl(\sum_{j=1}^\infty \norm{M_{\varphi_j,t}f}_p^p\Biggr)^{1/p}.
  \end{align*}
  Thus, there exists $j\in \N$ such that
  \[\norm{BM_{\varphi_j,t} f}_p \geq \bigl(\norm{B|_F} - \varepsilon\bigr)\norm{M_{\varphi_j,t}f}_p.\]
  Since $\spt \left(M_{\varphi_{j,t}}f\right) \subseteq B[x,r_t]\cap F$ for some $x\in X$, we get
  \[\tripnorm{B|_F}_t \geq \norm{B|_F} - \varepsilon.\]
  The other inequality is clear by definition.
\end{proof}

\begin{theorem}
\label{thm:norm_equivalence}
  Let $A\in\calA^p$. Then
  \[\frac{1}{\norm{P}} \norm{A+\calK(M^p)} \leq \sup\limits_{x\in\bdry{X}} \norm{A_x} \leq \norm{A+\calK(M^p)}.\]
  In particular, $\norm{A+\calK(M^p)} = \sup\limits_{x\in\bdry{X}} \norm{A_x}$ if $\norm{P} = 1$.
\end{theorem}

\begin{proof}
  Let $x\in\bdry{X}$. For $K\in\calK(M^p)$ we obtain $K_x=0$ by Proposition \ref{prop:limit_operators_of_compact_operators}.
  Thus, by Proposition \ref{prop:limit_operator_properties} we observe
  \[\norm{A_x} = \norm{A_x+K_x} = \norm{(A+K)_x} \leq \norm{A+K}.\]
  Hence,
  \[\sup_{x\in\bdry{X}} \norm{A_x} \leq \inf_{K\in \calK(M^p)} \norm{A+K} = \norm{A+\calK(M^p)}.\]
  
  Concerning the first inequality, for $K\in\calK\bigl(L_p(X,\mu),M^p\bigr)$ we estimate
  \begin{align*}
    \norm{AP+K}_{\calL\bigl(L_p(X,\mu),M^p\bigr)} & = \sup_{\genfrac{}{}{0pt}{}{f\in L_p(X,\mu)}{\norm{f}_p = 1}} \norm{(AP+K)f}_p
    \geq \sup_{\genfrac{}{}{0pt}{}{f\in M^p}{\norm{f}_p = 1}} \norm{(AP+K)f}_p \\
    & = \sup_{\genfrac{}{}{0pt}{}{f\in M^p}{\norm{f}_p = 1}} \norm{(A+K)f}_p
    = \norm{A+K|_{M^p}}_{\calL(M^p)}.
  \end{align*}
  Thus,
  \[\norm{A+\calK(M^p)} \leq \inf_{K\in\calK\bigl(L_p(X,\mu),M^p\bigr)} \norm{AP+K}.\]
  It is therefore sufficient to show that
  \[\inf_{K\in\calK\bigl(L_p(X,\mu),M^p\bigr)} \norm{AP+K} \leq \sup_{x\in\bdry{X}} \norm{A_xP_x}.\]
  Assume by contradiction that there exists $\varepsilon>0$ such that
  \[\inf_{K\in\calK\bigl(L_p(X,\mu),M^p\bigr)} \norm{AP+K} > \sup_{x\in\bdry{X}} \norm{A_xP_x} + \varepsilon.\]
  Let $x_0\in X$ be as in Assumption \ref{ass:shifts}. Then, in particular, for all $s>0$ we have
  \[\norm{AP|_{X\setminus B[x_0,s]}} = \norm{APM_{\1_{X\setminus B[x_0,s]}}} = \norm{AP - APM_{\1_{B[x_0,s]}}} > \sup_{x\in\bdry{X}} \norm{A_xP_x} + \varepsilon,\]
  since $B[x_0,s]$ is compact by Assumption \ref{ass:space} and therefore $PM_{\1_{B[x_0,s]}} \in \calK\bigl(L_p(X,\mu),M^p\bigr)$ by Assumption \ref{ass:shifts}.
  By Proposition \ref{prop:local_norms}, there exists $t>0$ such that for all $s>0$ we have
  \[\tripnorm{AP|_{X\setminus B[x_0,s]}}_t \geq \norm{AP|_{X\setminus B[x_0,s]}} - \frac{1}{4}\varepsilon > \sup_{x\in\bdry{X}} \norm{A_xP_x} + \frac{3}{4}\varepsilon.\]
  Hence, for all $s>0$ there exists $x_s\in X$ such that
  \[\norm{APM_{\1_{B[x_s,r_t]}}} \geq \norm{APM_{\1_{B[x_s,r_t]\setminus B[x_0,s]}}} \geq \tripnorm{AP|_{X\setminus B[x_0,s]}}_t -\frac{1}{4}\varepsilon > \sup_{x\in\bdry{X}} \norm{A_xP_x} + \frac{1}{2}\varepsilon.\]
  Note that $(x_s)$ cannot be bounded as the second term above would be $0$ for sufficiently large $s$. Since $\beta X$ is compact there exists a subnet of $(x_s)$, again denoted by $(x_s)$, and $x\in \bdry{X}$ such that $x_s\to x$.
  By Lemma \ref{lem:surjective_isometry} and Lemma \ref{lem:multiplication_and_shifts}, we have
  \[\norm{U_{x_s}^p AP (U_{x_s}^p)^{-1} M_{\1_{B[x_0,r_t]}}} = \norm{A P M_{\1_{B[x_s,r_t]}}}
    > \sup_{x\in\bdry{X}} \norm{A_xP_x} + \frac{1}{2}\varepsilon.\]
  By Theorem \ref{thm:limit_operator_existence}, $\bigl(U_{x_s}^p AP (U_{x_s}^p)^{-1} M_{\1_{B[x_0,r_t]}}\bigr)$ converges to $A_xP_xM_{\1_{B[x_0,r_t]}}$. Thus,
  \[\norm{U_{x_s}^p AP (U_{x_s}^p)^{-1} M_{\1_{B[x_0,r_t]}}}\to \norm{A_xP_x M_{\1_{B[x_0,r_t]}}},\]
  and therefore
  \[\norm{A_x P_x M_{\1_{B[x_0,r_t]}}} \geq \sup_{x\in\bdry{X}} \norm{A_xP_x} + \frac{1}{2}\varepsilon,\]
  which is clearly a contradiction. 
\end{proof}

By means of Theorem \ref{thm:norm_equivalence} we can now show the converse of Proposition \ref{prop:limit_operators_of_compact_operators}.

\begin{corollary}
\label{cor:characterization_compact_operators_limit_operators}
  Let $K\in\calL(M^p)$. Then $K$ is compact if and only if $K \in \calA^p$ and $K_x=0$ for all $x\in\bdry{X}$.
\end{corollary}

\begin{proof}
  If $K\in\calK(M^p)$, then Theorem \ref{thm:Ap_properties} implies $K \in \calA^p$ and Proposition \ref{prop:limit_operators_of_compact_operators} (or Theorem \ref{thm:norm_equivalence}) yields $K_x=0$ for all $x\in\bdry{X}$. On the other hand, if $K \in \calA^p$ and $K_x=0$ for all $x\in\bdry{X}$, then Theorem \ref{thm:norm_equivalence} yields $\norm{K+\calK(M^p)} = 0$, i.e.\ $K\in\calK(M^p)$.
\end{proof}

\subsection{Fredholm operators}

In this subsection we show (II) and (III). We start with a few preliminary results.
We start with a convergence lemma (including a norm bound) which originally (i.e.,\ for $\ell_p$-spaces over $\Z^N$) goes back to Simonenko \cite{Simonenko1968}, see also \cite[Proposition 2.2.2]{RabinovichRochSilbermann2004} and \cite[Lemma 3.14]{Lindner2006}.

\begin{lemma}
\label{lem:series_strong_convergence}
  Let $J \subseteq \N$ and assume that the set $\set{A_j : j\in J}\subseteq \calL\bigl(L_p(X,\mu)\bigr)$ is bounded. Then for every  $t>0$ the series $\sum\limits_{j\in J} M_{\varphi_{j,t}^{p/q}} A_j M_{\varphi_{j,t}}$ converges $*$-strongly and
  \[\Biggl\|\sum_{j\in J} M_{\varphi_{j,t}^{p/q}} A_j M_{\varphi_{j,t}}\Biggr\|_p \leq \sup_{j\in J}\norm{A_j}.\]
\end{lemma}

\begin{proof}
  Let $f \in L_p(X,\mu)$. Using Lemma \ref{lem:sum_estimate}, we obtain
  \[\Biggl\|\sum_{j\in J} M_{\varphi_{j,t}^{p/q}} A_j M_{\varphi_{j,t}}f\Biggr\|_p \leq \Biggl(\sum\limits_{j \in J} \norm{A_jM_{\varphi_{j,t}}f}_p^p\Biggr)^{1/p} \leq \sup\limits_{j \in J} \norm{A_j} \Biggl(\sum\limits_{j \in J} \norm{M_{\varphi_{j,t}}f}_p^p\Biggr)^{1/p} \leq \sup_{j\in J}\norm{A_j}\norm{f}_p.\]
  This implies the strong convergence and the norm estimate. The convergence of the adjoints is shown analogously.
\end{proof}

The next proposition is similar to \cite[Proposition 2.2.3]{RabinovichRochSilbermann2004} and \cite[Proposition 3.15]{Lindner2006}, and relies on Lemma \ref{lem:series_strong_convergence}.

\begin{proposition}
\label{prop:Fredholmness}
  Let $A\in\BDO^p$, $[A,P] = 0$, $c>0$ and $t_0 > 0$ such that for all $t \leq t_0$ there exists $j_0\in\N$ such that for all $j\geq j_0$ there exist $B_{j,t},C_{j,t}\in \calL\bigl(L_p(X,\mu)\bigr)$ with $\norm{B_{j,t}},\norm{C_{j,t}}\leq c$ and
  \[B_{j,t}AM_{\varphi_{j,t}} = M_{\varphi_{j,t}}, \quad M_{\varphi_{j,t}^{p/q}}AC_{j,t} = M_{\varphi_{j,t}^{p/q}}.\]
  Then $A|_{M^p}\in \calL(M^p)$ is Fredholm and $\norm{(A|_{M^p} + \calK(M^p))^{-1}} \leq 2\norm{P} c$.
\end{proposition}

\begin{proof}
  Let $t>0$. By Lemma \ref{lem:series_strong_convergence}, the series
  \[B_t:=\sum_{j\geq j_0} M_{\varphi_{j,t}^{p/q}} B_{j,t} M_{\varphi_{j,t}}=\sum_{j\geq j_0} M_{\varrho_{j,t}^{1/q}} B_{j,t} M_{\varrho_{j,t}^{1/p}}\]
  converges strongly with $\norm{B_t}\leq c$. Hence,
  \[B_tA = \sum_{j\geq j_0} M_{\varphi_{j,t}^{p/q}} B_{j,t} M_{\varphi_{j,t}} A = \sum_{j\geq j_0} M_{\varphi_{j,t}^{p/q}} B_{j,t} AM_{\varphi_{j,t}} + \sum_{j\geq j_0} M_{\varphi_{j,t}^{p/q}} B_{j,t} [M_{\varphi_{j,t}}, A].\]
  For the second term we have, using Lemma \ref{lem:sum_estimate},
  \begin{align*}
  \Biggl\|\sum_{j\geq j_0} M_{\varphi_{j,t}^{p/q}} B_{j,t} [M_{\varphi_{j,t}}, A]\Biggr\| &= \sup\limits_{\norm{f}_p = 1} \Biggl\|\sum_{j\geq j_0} M_{\varphi_{j,t}^{p/q}} B_{j,t} [M_{\varphi_{j,t}}, A]f\Biggr\|_p \leq \sup\limits_{\norm{f}_p = 1} \Biggl(\sum\limits_{j \geq j_0} \norm{B_{j,t}[M_{\varphi_{j,t}},A]f}_p^p\Biggr)^{1/p}\\ 
  &\leq c\sup\limits_{\norm{f}_p = 1} \Biggl(\sum\limits_{j \geq j_0} \norm{[M_{\varphi_{j,t}},A]f}_p^p\Biggr)^{1/p}
  \end{align*}
  which tends to $0$ as $t \to 0$ by Proposition \ref{prop:BDO_commutator}. For the first term, we obtain
  \[\sum_{j\geq j_0} M_{\varphi_{j,t}^{p/q}} B_{j,t} AM_{\varphi_{j,t}} = \sum_{j\geq j_0} M_{\varphi_{j,t}^{p/q}} M_{\varphi_{j,t}} = \sum_{j\geq j_0} M_{\varphi_{j,t}^p}.\]
  Hence,
  \[\Biggl\|B_t A - \sum_{j\geq j_0} M_{\varrho_{j,t}}\Biggr\| = \Biggl\|B_t A - \sum_{j\geq j_0} M_{\varphi_{j,t}^p}\Biggr\| \to 0\]
  as $t \to 0$. In particular,
  \[\Biggl\|P B_t A|_{M^p} - \sum_{j\geq j_0} P M_{\varrho_{j,t}}|_{M^p}\Biggr\| \leq \norm{P} \Biggl\|B_t A - \sum_{j\geq j_0} M_{\varrho_{j,t}}\Biggr\| \to 0\]
  as $t \to 0$. As the functions $\varrho_{j,t}$ have bounded support, the operators $PM_{\varrho_{j,t}}|_{M^p}$ are compact for all $j\in \N$ by Assumption \ref{ass:subspaces_and_projection}. Since $\sum\limits_{j = 1}^{\infty} \varrho_{j,t} = 1$ for all $t>0$, we have
  \[\sum_{j\geq j_0} P M_{\varrho_{j,t}}|_{M^p} = \sum_{j \in \N} P M_{\varrho_{j,t}}|_{M^p} - \sum_{j<j_0} P M_{\varrho_{j,t}}|_{M^p} = I_{M^p} - \sum_{j<j_0} P M_{\varrho_{j,t}}|_{M^p} \in I+\calK(M^p).\]
  Therefore, by a Neumann series argument, there exists a $B\in \calL(M^p)$ such that $BA|_{M^p} \in I+\calK(M^p)$ and
  \[\norm{B+\calK(M^p)} \leq 2\norm{P}\sup_{t>0} \norm{B_t} \leq 2 \norm{P} c.\]
  
  By Theorem \ref{thm:BDO_algebraic}, $A^* \in \BDO^q$ for $\frac{1}{p} + \frac{1}{q} = 1$. Therefore, noting that $\varphi_{j,t}^{p/q} = \varrho_{j,t}^{1/q}$, we may apply the above to $A^*$ to obtain an operator
  \[C_t^* := \sum_{j\geq j_0} M_{\varrho_{j,t}^{1/p}} C_{j,t}^* M_{\varrho_{j,t}^{1/q}} = \sum_{j\geq j_0} M_{\varphi_{j,t}} C_{j,t}^* M_{\varphi_{j,t}^{p/q}}\]
  with
  \[\Biggl\|AC_t - \sum_{j\geq j_0} M_{\varrho_{j,t}}\Biggr\| = \Biggl\|C_t^*A^* - \sum_{j\geq j_0} M_{\varrho_{j,t}}\Biggr\| \to 0.\]
  Hence
  \[\Biggl\|APC_t|_{M^p} - \sum_{j\geq j_0} P M_{\varrho_{j,t}}|_{M^p}\Biggr\| \leq \norm{P} \Biggl\|AC_t - \sum_{j\geq j_0} M_{\varrho_{j,t}}\Biggr\| \to 0,\]
  using $[A,P] = 0$. As above, this implies that there exists a $C\in \calL(M^p)$ such that $A|_{M^p}C \in I+\calK(M^p)$ and
  \[\norm{C+\calK(M^p)} \leq 2\norm{P}\sup_{t>0} \norm{C_t} \leq 2 \norm{P} c.\]
  Thus $A|_{M^p}$ is Fredholm and $\norm{(A|_{M^p} + \calK(M^p))^{-1}} \leq 2\norm{P} c$.
\end{proof}

Recall $Q = I-P$ and $\hat{A} = AP+Q$ for $A \in \calL(M^p)$.

\begin{proposition}
\label{prop:shifts}
  Let $A\in \calA^p$, $(x_\iota)$ a net in $X$ with $x_\iota\to x\in\bdry{X}$, $A_x$ invertible and $f\in L_{\infty,\rm c}(X,\mu)$. 
  Then there exists $\iota_0$ such that for all $\iota\geq \iota_0$ there exist operators
  $B_\iota,C_\iota\in \calL\bigl(L_p(X,\mu)\bigr)$ with $\norm{B_\iota},\norm{C_\iota}\leq 2(\norm{A_x^{-1}}\norm{P}+\norm{Q})$ and
  \[B_\iota \hat{A} M_{f\circ \phi_{x_\iota}^{-1}} = M_{f\circ \phi_{x_\iota}^{-1}} = M_{f\circ\phi_{x_\iota}^{-1}} \hat{A}C_\iota.\]
\end{proposition}

\begin{proof}
  Let $K\subseteq X$ be a compact set that contains $\spt f$. As $P \in \BDO^p$ by Assumption \ref{ass:subspaces_and_projection}, $Q$ is also in $\BDO^p$. By Corollary \ref{cor:hats}, we have
  \[\lim\limits_{x_{\iota} \to x} \norm{\Bigl(U_{x_\iota}^p (AP+Q)(U_{x_\iota}^p)^{-1} - (A_xP_x+Q_x)\Bigr)M_{\1_K}} = 0.\]
  Since $A_x$ is invertible, $A_xP_x+Q_x$ is also invertible with
  \[(A_xP_x+Q_x)^{-1} = A_x^{-1}P_x+Q_x.\]
  Hence, there exists $\iota_0$ such that
  \[R_\iota:=(A_xP_x+Q_x)^{-1}\Bigl(U_{x_\iota}^p (AP+Q)(U_{x_\iota}^p)^{-1} - (A_xP_x+Q_x)\Bigr)M_{\1_K}\]
  satisfies $\norm{R_\iota}\leq \frac{1}{2}$ for all $\iota\geq \iota_0$. In particular, $I+R_\iota\in \calL\bigl(L_p(X,\mu)\bigr)$ is invertible for all $\iota\geq\iota_0$.
  Using
  \[U_{x_\iota}^p (AP+Q)(U_{x_\iota}^p)^{-1} M_{\1_K} = (A_xP_x+Q_x)M_{\1_K} + (A_xP_x+Q_x)R_\iota,\]
  we get
  \[(A_xP_x+Q_x)^{-1}U_{x_\iota}^p (AP+Q)(U_{x_\iota}^p)^{-1} M_f = (I+R_\iota)M_f,\]
  and thus
  \[(I+R_\iota)^{-1} (A_xP_x+Q_x)^{-1}U_{x_\iota}^p (AP+Q)(U_{x_\iota}^p)^{-1} M_f = M_f.\]
  Applying $(U_{x_\iota}^p)^{-1}$ from the left and $U_{x_\iota}^p$ from the right, by Lemma \ref{lem:multiplication_and_shifts}, we obtain
  \[(U_{x_\iota}^p)^{-1}(I+R_\iota)^{-1} (A_xP_x+Q_x)^{-1}U_{x_\iota}^p (AP+Q) M_{f\circ\phi_{x_\iota}^{-1}} = M_{f\circ\phi_{x_\iota}^{-1}}.\]
  
  Similarly,
  \[\lim\limits_{x_{\iota} \to x} \norm{M_{\1_K}\Bigl(U_{x_\iota}^p (AP+Q)(U_{x_\iota}^p)^{-1} - (A_xP_x+Q_x)\Bigr)} = 0.\]
  Hence,
  \[M_{f\circ\phi_{x_\iota}^{-1}} (AP+Q) U_{x_\iota}^p (A_xP_x+Q_x)^{-1} (I+S_\iota)^{-1}(U_{x_\iota}^p)^{-1} = M_{f\circ\phi_{x_\iota}^{-1}}\]
  for sufficiently large $\iota$, where
  \[S_\iota:= M_{\1_{K}}\Bigl(U_{x_\iota}^p (AP+Q)(U_{x_\iota}^p)^{-1} - (A_xP_x+Q_x)\Bigr)(A_xP_x+Q_x)^{-1}. \qedhere\]
\end{proof}

\begin{theorem}
\label{thm:invertiblility_and_uniform_bound_implies_Fredholm}
  Let $A\in\calA^p$. If $A_x$ invertible for all $x\in\bdry{X}$ and $\sup\limits_{x\in\bdry{X}} \norm{A_x^{-1}}<\infty$, then $A$ is Fredholm.
\end{theorem}

\begin{proof}
  Assume that $A$ is not Fredholm. 
  Note that
  \[[\hat{A},P] = (AP+Q)P - P(AP+Q) = AP-PAP = 0.\]
  By Proposition \ref{prop:Fredholmness} there exists a $t>0$ and an increasing sequence $(j_m)_{m \in \N}$ in $\N$ such that
  \[B\hat{A}M_{\varphi_{j_m,t}} \neq M_{\varphi_{j_m,t}} \quad\text{ or }\quad M_{\varphi_{j_m,t}^{p/q}}\hat{A}B \neq M_{\varphi_{j_m,t}^{p/q}}\]
  for all $m\in \N$ and all $B\in \calL\bigl(L_p(X,\mu)\bigr)$ with $\norm{B}\leq 2\Bigl(\sup\limits_{x\in\bdry{X}} \norm{A_x^{-1}}\norm{P} + \norm{Q}\Bigr)$.
  W.l.o.g.\ (by only considering one of the two statements and choosing an approriate subsequence) we may assume that
  \[B\hat{A}M_{\varphi_{j_m,t}} \neq M_{\varphi_{j_m,t}} \quad (m\in\N).\]
  There exists $C>0$ such that $\diam\spt \varphi_{j,t}\leq C$ for all $j\in \N$ by Lemma \ref{lem:phis}. 
  Therefore there exist $(x_m)$ in $X$ and $R>0$ such that
  \[x_m\in \spt \varphi_{j_m,t}\subseteq B(x_m,R).\]
  As $(j_m)_{m \in \N}$ is increasing, Lemma \ref{lem:phis}(iv) implies $d(x,x_m) \to \infty$ for all $x \in X$. Since $\beta X$ is compact, there exists a subnet $(x_{m_\iota})_\iota$ of $(x_m)$ such that $x_{m_\iota}\to x$ for some $x\in \bdry{X}$. Let $x_0 \in X$ as in Assumption \ref{ass:shifts}, i.e.~$\phi_x(x_0) = x$ for all $x \in X$. By Proposition \ref{prop:shifts} there exists $\iota_0$ such that for all $\iota\geq \iota_0$ there exists an operator $B_\iota\in \calL\bigl(L_p(X,\mu)\bigr)$ such that $\norm{B_\iota}\leq 2(\norm{A_x^{-1}}\norm{P} + \norm{Q})$ and
  \[B_\iota \hat{A} M_{\1_{B(x_{m_\iota},R)}} = B_\iota \hat{A} M_{\1_{B(x_{0},R)}\circ\phi_{x_{m_\iota}}^{-1}} = M_{\1_{B(x_{0},R)}\circ\phi_{x_{m_\iota}}^{-1}} = M_{\1_{B(x_{m_\iota},R)}}.\]
  Multiplying by $M_{\varphi_{j_{m_\iota},t}}$ from the right yields
  \[B_\iota \hat{A} M_{\varphi_{j_{m_\iota},t}} = M_{\varphi_{j_{m_\iota},t}},\]
  for all $\iota\geq \iota_0$, a contradiction. 
\end{proof}

The converse of Theorem \ref{thm:invertiblility_and_uniform_bound_implies_Fredholm} is also true and easier to prove.

\begin{theorem}
\label{thm:Fredholm_implies_lim-ops_invertible}
  Let $A\in \calA^p$ be Fredholm and $x\in \Gamma X$. Then $A_x$ is invertible and $\norm{A_x^{-1}} \leq \norm{(A+\calK(M^p))^{-1}}$.
\end{theorem}

\begin{proof}
Let $B \in \calL(M^p)$ be a Fredholm regulariser of $A$, i.e.~$AB = I + K_1$ and $BA = I + K_2$ for two compact operators $K_1,K_2 \in \calK(M^p)$. By Theorem \ref{thm:Ap_properties}, $B\in \calA^p$. Using Proposition \ref{prop:limit_operator_properties} and Proposition \ref{prop:limit_operators_of_compact_operators}, we get
\[A_xB_x = (AB)_x = (I+K_1)_x = I\]
and similarly, $B_xA_x = I$. Hence, $A_x$ is invertible and $\norm{A_x^{-1}} = \norm{B_x} \leq \norm{B}$ by Proposition \ref{prop:limit_operator_properties} again. As this is true for every regulariser $B$, we get $\norm{A_x^{-1}} \leq \norm{(A+\calK(M^p))^{-1}}$.
\end{proof}

\subsection{Lower norm}

\begin{definition}
  Let $Y_1, Y_2$ be complex Banach spaces and $A\in\calL\bigl(Y_1,Y_2\bigr)$. Then
  \[\nu(A):=\inf\set{\norm{Ax}_{Y_2} :\; x\in Y_1,\, \norm{x} = 1}\]
  is called the \emph{lower norm} of $A$.
\end{definition}

\begin{lemma}[{see \cite[Lemma 2.35]{Lindner2006}}]
  Let $Y_1, Y_2$ be complex Banach spaces and $A\in\calL\bigl(Y_1,Y_2\bigr)$ an invertible operator. Then
  $\nu(A) = \norm{A^{-1}}^{-1}$.
\end{lemma}

For our specific setup of operators on $L_p$-spaces we will need two refined notions.
For $t>0$ let $r_t:=\sup\limits_{j\in\N} \diam \spt \varphi_{j,t}$ as above, which is finite by assumption.
\begin{definition}
  For $t>0$, $F\subseteq X$ a Borel set and $A\in \calL\bigl(L_p(X,\mu)\bigr)$ we define
  \begin{align*}
    \nu_t(A|_F) & := \inf\set{\norm{Af}_p:\; f\in L_p(X,\mu),\, \norm{f}_p=1,\, \spt f \subseteq B[x,r_t]\cap F\,\text{for some $x\in X$}},\\
    \nu(A|_F) & := \inf\set{\norm{Af}_p:\; f\in L_p(X,\mu),\, \norm{f}_p=1,\,\spt f\subseteq F}.
  \end{align*}
\end{definition}
Note that $\nu(A) = \nu(A|_X)$. 

\begin{lemma}
\label{lem:lower_norm_of_limit_operators}
  Let $A \in \calA^p$ and $x \in \beta X$. Then $\nu(A_x) \geq \nu(A)$.
\end{lemma}

\begin{proof}
  Let $(x_\iota)_\iota$ be a net in $X$ that converges to $x \in \beta X$ and $\varepsilon > 0$.
  Moreover, choose $f \in \ran(P_x)$ with $\norm{f}_p = 1$ such that $\norm{A_xf}_p < \nu(A_x) + \varepsilon$.
  By Theorem \ref{thm:limit_operator_existence} and Corollary \ref{cor:strong_convergence}, we can choose $\iota$ sufficiently large such that $\norm{\left(A_x - U_{x_\iota}^pAP(U_{x_\iota}^p)^{-1}\right)f}_p < \varepsilon$ and $\norm{U_{x_\iota}^pP(U_{x_\iota}^p)^{-1}f - f}_p < \varepsilon$. Using that $U_{x_\iota}^p$ is an isometry, we get
  \[\nu(A_x) > \norm{A_xf}_p - \varepsilon > \norm{U_{x_\iota}^pAP(U_{x_\iota}^p)^{-1}f}_p - 2\varepsilon \geq \nu(A)\norm{U_{x_\iota}^pP(U_{x_\iota}^p)^{-1}f}_p - 2\varepsilon > \nu(A)\norm{f}_p - 3\varepsilon.\]
  As $\varepsilon$ was arbitrary, $\nu(A_x) \geq \nu(A)$.
\end{proof}

\begin{lemma}
\label{lem:estimate_lower_norm}
  Let $A,B\in \calL\bigl(L_p(X,\mu)\bigr)$ and $F\subseteq X$ a Borel set. Then:
  \begin{enumerate}
    \item
      $\abs{\nu(A|_F) - \nu(B|_F)} \leq \norm{(A-B)M_{\1_F}}$.
    \item
      For $t>0$: $\abs{\nu_t(A|_F) - \nu_t(B|_F)} \leq \norm{(A-B)M_{\1_F}}$.
  \end{enumerate}
\end{lemma}

\begin{proof}
  We only prove (a), the same proof also works for (b).
  Let $\varepsilon>0$. There exists $f\in L_p(X,\mu)$ with $\norm{f}_p = 1$, $\spt f\subseteq F$ and $\norm{Bf}_p\leq \nu(B|_F) + \varepsilon$.
  Hence,
  \[\nu(A|_F) - \nu(B|_F) - \varepsilon \leq \nu(A|_F) - \norm{Bf}_p \leq \norm{Af}_p - \norm{Bf}_p \leq \norm{(A-B)f}_p \leq \norm{(A-B)M_{\1_F}}.\]
  Similarly,
  \[\nu(B|_F) - \nu(A|_F) - \varepsilon \leq \norm{(A-B)M_{\1_F}}.\]
  Therefore, $\abs{\nu(A|_F) - \nu(B|_F)} \leq \norm{(A-B)M_{\1_F}}$.  
\end{proof}

The next lemma is the analogue of Proposition \ref{prop:local_norms} for the lower norm.

\begin{lemma}
\label{lem:lower_norm_limit_operator}
  Let $A\in \BDO^p$ and $\varepsilon>0$. Then there is a $t_0>0$ such that for all $t \leq t_0$, all Borel sets $F\subseteq X$ and all operators $B\in \{A\}\cup \set{A_x:\; x\in \beta X}$ we have
  \[\nu(B|_F) \leq \nu_t(B|_F) \leq \nu(B|_F) + \varepsilon.\]
\end{lemma}

\begin{proof}
  The first inequality is clear. For the second equality let $\varepsilon > 0$, $B\in \set{A}\cup\set{A_x:\; x\in \beta X}$ and $F\subseteq X$ a Borel set. Let $f\in L_p(X,\mu)$ with $\norm{f}_p = 1$ and $\spt f\subseteq F$, such that 
  \[\norm{Bf}_p \leq \nu(B|_F) + \frac{1}{2}\varepsilon.\]
  By Lemma \ref{lem:BDO_commutator_uniform} there is a $t_0>0$ (independent of $B$ and $f$) such that for all $t \leq t_0$:
  \[\Biggl(\sum_{j=1}^\infty \norm{[B, M_{\varphi_{j,t}}]f}_p^p\Biggr)^{1/p} \leq \frac{1}{2}\varepsilon.\]
  Since $\sum\limits_{j=1}^\infty \varphi_{j,t}(x)^p = 1$ for all $x\in X$, we have
  \[\Biggl(\sum_{j=1}^\infty \norm{M_{\varphi_j,t} Bf}_p^p\Biggr)^{1/p} = \norm{Bf}_p \leq \nu(B|_F) + \frac{1}{2}\varepsilon.\]
  Minkowski's inequality yields
  \begin{align*}
    \Biggl(\sum_{j=1}^\infty \norm{BM_{\varphi_j,t} f}_p^p\Biggr)^{1/p} 
    & \leq \Biggl(\sum_{j=1}^\infty \norm{M_{\varphi_j,t} Bf}_p^p\Biggr)^{1/p} + \Biggl(\sum_{j=1}^\infty \norm{[B,M_{\varphi_{j,t}}]f}_p^p\Biggr)^{1/p} \\
    & \leq \nu(B|_F) + \varepsilon\\
    &= \bigl(\nu(B|_F) + \varepsilon\bigr) \Biggl(\sum_{j=1}^\infty \norm{M_{\varphi_j,t}f}_p^p\Biggr)^{1/p}.
  \end{align*}
  Thus, there exists $j\in \N$ such that
  \[\norm{BM_{\varphi_j,t} f}_p \leq \bigl(\nu(B|_F) + \varepsilon\bigr)\norm{M_{\varphi_j,t}f}_p.\]
  Since $\spt \left(M_{\varphi_{j,t}}f\right) \subseteq B[x,r_t]\cap F$ for some $x\in X$, we get
  \[\nu_t(B|_F) \leq \nu(B|_F) + \varepsilon.\qedhere\]
\end{proof}

As above we use the notation $\hat{A}_x = A_xP_x+Q_x$ for $x \in \beta X$.

\begin{lemma}
\label{lem:lower_norm_localization}
  Let $A\in\calA^p$, $x \in \Gamma X$, $w \in X$ and $r > 0$. Then there exists $y \in \Gamma X$ such that
  \begin{equation} \label{eq:lower_norm_localization}
  \nu(\hat{A}_y|_{B[x_0,r]}) \leq \nu(\hat{A}_x|_{B[w,r]}) \quad \text{and} \quad \nu(\hat{A}_y|_{B[x_0,r+d(x_0,w)]}) \leq \nu(\hat{A}_x|_{B[x_0,r]}).
  \end{equation}
\end{lemma}

\begin{proof}
  We divide the proof into two steps. In step (i) we show that \eqref{eq:lower_norm_localization} holds up to an error of $\varepsilon$. With a compactness argument we take $\varepsilon \to 0$ in step (ii).

  (i) Let $\varepsilon > 0$. Choose a net $(x_{\iota})$ in $X$ that converges to $x$. As $\beta X$ is compact, we may assume that the net $(\phi_{x_{\iota}}(w))$ also converges, say to some $y \in \beta X$. As $X$ is proper, we must in fact have $y \in \Gamma X$. Indeed, otherwise $x_{\iota} \in B[y,2d(w,x_0)]$ for sufficiently large $\iota$ by
  \[d(\phi_{x_{\iota}}(w),x_{\iota}) = d(\phi_{x_{\iota}}(w),\phi_{x_{\iota}}(x_0)) = d(w,x_0).\]
  This would imply $x \in B[y,2d(w,x_0)] \subseteq X$, which contradicts $x\in \Gamma X$. Choose $f\in L_p(X,\mu)$ with $\norm{f}_p = 1$, $\spt f\subseteq B[w,r]$ and $\|\hat{A}_xf\| \leq \nu(\hat{A}_x|_{B[w,r]}) + \frac{\varepsilon}{3}$. Now choose $\iota$ sufficiently large such that
  \[\norm{\left(U_{x_{\iota}}^p\hat{A}(U_{x_{\iota}}^p)^{-1} - \hat{A}_x\right)f}_p \leq \frac{\varepsilon}{3} \quad \text{and} \quad \norm{\left(U_{\phi_{x_{\iota}}(w)}^p\hat{A}(U_{\phi_{x_{\iota}}(w)}^p)^{-1} - \hat{A}_y\right)M_{\1_{B[x_0,r]}}} \leq \frac{\varepsilon}{3}.\]
  Note that this is possible by Corollary \ref{cor:hats}. Set $g := U_{\phi_{x_{\iota}}(w)}^ p(U_{x_{\iota}}^p)^{-1}f$. As $U_{\phi_{x_{\iota}}(w)}^p$ and $(U_{x_{\iota}}^p)^{-1}$ are isometries, we have $\norm{g}_p = \norm{f}_p = 1$. Moreover,
  \[\spt g \subseteq \spt(f \circ \phi_{x_{\iota}}^{-1} \circ \phi_{\phi_{x_{\iota}}(w)}) \subseteq \left(\phi_{\phi_{x_{\iota}}(w)}^{-1} \circ \phi_{x_{\iota}}\right)(B[w,r]) = B[x_0,r]\]
  and
  \begin{align*}
  \|\hat{A}_yg\|_p &\leq \norm{\left(\hat{A}_y - U_{\phi_{x_{\iota}}(w)}^p\hat{A}(U_{\phi_{x_{\iota}}(w)}^p)^{-1}\right)g}_p + \norm{U_{\phi_{x_{\iota}}(w)}^p\hat{A}(U_{\phi_{x_{\iota}}(w)}^p)^{-1}g}_p\\
  &\leq \norm{\left(\hat{A}_y - U_{\phi_{x_{\iota}}(w)}^p\hat{A}(U_{\phi_{x_{\iota}}(w)}^p)^{-1}\right)M_{\1_{B[x_0,r]}}} + \norm{U_{x_{\iota}}^p\hat{A}(U_{x_{\iota}}^p)^{-1}f}_p\\
  &\leq \frac{2\varepsilon}{3} + \|\hat{A}_x f\|_p\\
  &\leq \varepsilon + \nu(\hat{A}_x|_{B[w,r]}),
  \end{align*}
  i.e.~$\nu(\hat{A}_y|_{B[x_0,r]}) \leq \nu(\hat{A}_x|_{B[w,r]}) + \varepsilon$. 
  
  Now let $f\in L_p(X,\mu)$ with $\norm{f}_p = 1$, $\spt f\subseteq B[x_0,r]$ and $\|\hat{A}_xf\| \leq \nu(\hat{A}_x|_{B[x_0,r]}) + \frac{\varepsilon}{3}$. Again, we choose $g := U_{\phi_{x_{\iota}}(w)}^ p(U_{x_{\iota}}^p)^{-1}f$, where $\iota$ is sufficiently large such that
  \[\norm{\left(U_{x_{\iota}}^p\hat{A}(U_{x_{\iota}}^p)^{-1} - \hat{A}_x\right)f}_p \leq \frac{\varepsilon}{3} \quad \text{and} \quad \norm{\left(U_{\phi_{x_{\iota}}(w)}^p\hat{A}(U_{\phi_{x_{\iota}}(w)}^p)^{-1} - \hat{A}_y\right)M_{\1_{B[x_0,r+d(x_0,w)]}}} \leq \frac{\varepsilon}{3}.\]
  Since $d(\phi_{\phi_{x_{\iota}}(w)}^{-1}(\phi_{x_{\iota}}(x_0)),x_0) = d(\phi_{x_{\iota}}(x_0),\phi_{\phi_{x_{\iota}}(w)}(x_0)) = d(x_0,w)$, we get
  \[\spt g \subseteq \left(\phi_{\phi_{x_{\iota}}(w)}^{-1} \circ \phi_{x_{\iota}}\right)(B[x_0,r]) \subseteq B[x_0,r+d(x_0,w)].\]
  Moreover, we have $\|\hat{A}_yg\|_p \leq \varepsilon + \nu(\hat{A}_x|_{B[x_0,r]})$ by the same calculation as above. Therefore we obtain $\nu(\hat{A}_y|_{B[x_0,r+d(x_0,w)]}) \leq \nu(\hat{A}_x|_{B[x_0,r]}) + \varepsilon$.
  
  (ii) For $n \in \N$ set $\varepsilon := \frac{1}{n}$ in (i). This yields a sequence $(y_n)_{n \in \N}$ in $\Gamma X$ with
  \begin{equation} \label{eq:lower_norm_localization2}
  \nu(\hat{A}_{y_n}|_{B[x_0,r]}) \leq \nu(\hat{A}_x|_{B[w,r]}) + \frac{1}{n} \quad \text{and} \quad \nu(\hat{A}_{y_n}|_{B[x_0,r+d(x_0,w)]}) \leq \nu(\hat{A}_x|_{B[x_0,r]}) + \frac{1}{n}.
  \end{equation}
  As $\Gamma X$ is compact, $(y_n)_{n \in \N}$ has a subnet $(y_{n_{\iota}})$ that converges to some $y \in \Gamma X$. Using Lemma \ref{lem:estimate_lower_norm} and Corollary \ref{cor:hats}, we get
  \begin{align*}
  \abs{\nu(\hat{A}_y|_{B[x_0,r]}) - \nu(\hat{A}_{y_{n_{\iota}}}|_{B[x_0,r]})} \leq \norm{\left(\hat{A}_y - \hat{A}_{y_{n_{\iota}}}\right)M_{\1_{B[x_0,r]}}} \to 0.
  \end{align*}
  Taking the limit in \eqref{eq:lower_norm_localization2} thus yields the result.
\end{proof}

\begin{proposition}
\label{prop:min_of_nu}
  Let $A\in\calA^p$. Then there exists $x\in \Gamma X$ such that
  \[\nu(\hat{A}_x) = \inf \set{\nu(\hat{A}_z):\; z\in \Gamma X}.\]
\end{proposition}

Proposition \ref{prop:min_of_nu} has its roots in \cite[Theorem 8]{LindnerSeidel2014}, where the corresponding statement was shown for suitable operators on $\ell_p$-spaces over $\Z^N$.

\begin{proof}
  By Lemma \ref{lem:lower_norm_limit_operator} we obtain a sequence $(t_k)_{k\in\N}$ with $r_{t_{k+1}}>2r_{t_k}$ and
  \begin{equation} \label{eq:min_of_nu}
  \nu_{t_k}(B|_F) < \nu(B|_F) + 2^{-(k+1)}
  \end{equation}
  for all $k\in\N$, all Borel sets $F\subseteq X$ and all $B\in\set{\hat{A}_x:\; x\in\bdry{X}}$.
  Choose a sequence $(x_n)_{n \in \N}$ in $\bdry{X}$ with
  \[\lim_{n\to\infty} \nu(\hat{A}_{x_n}) = \inf \set{\nu(\hat{A}_y):\; y\in \Gamma X}.\]
  By Lemma \ref{lem:lower_norm_localization} and \eqref{eq:min_of_nu} for $F = X$, for every $n \in \N$ we can find $w_n^0 \in X$ and $y_n^0 \in \Gamma X$ such that
  \[\nu(\hat{A}_{y_n^0}|_{B[x_0,r_{t_n}]}) \leq \nu(\hat{A}_x|_{B[w_n^0,r_{t_n}]}) \leq \nu(\hat{A}_x) + 2^{-(n+1)}.\]
  Next, using Lemma \ref{lem:lower_norm_localization} and \eqref{eq:min_of_nu} again with $F = B[x_0,r_{t_n}]$, we find $w_n^1 \in B[x_0,r_{t_n}+r_{t_{n-1}}]$ and $y_n^1 \in \Gamma X$ such that
  \begin{align*}
  \nu(\hat{A}_{y_n^1}|_{B[x_0,r_{t_{n-1}}]}) &\leq \nu(\hat{A}_{y_n^0}|_{B[w_n^1,r_{t_{n-1}}]}) \leq \nu(\hat{A}_{y_n^0}|_{B[w_n^1,r_{t_{n-1}}] \cap B[x_0,r_{t_n}]}) \leq \nu(\hat{A}_{y_n^0}|_{B[x_0,r_{t_n}]}) + 2^{-n}\\
  &\leq \nu(\hat{A}_{x_n}) + 2^{-n+1}.
  \end{align*}
  Repeating this argument, we can find $w_n^k \in B[x_0,r_{t_{n-k+1}}+r_{t_{n-k}}]$ and $y_n^k \in \Gamma X$ such that
  \begin{align*}
    \nu(\hat{A}_{y_n^k}|_{B[x_0,r_{t_{n-k}}]}) &\leq \nu(\hat{A}_{y_n^{k-1}}|_{B[w_n^k,r_{t_{n-k}}]}) \leq \nu(\hat{A}_{y_n^{k-1}}|_{B[w_n^k,r_{t_{n-k}}] \cap B[x_0,r_{t_{n-k+1}}]})\\
    &\leq \nu(\hat{A}_{y_n^{k-1}}|_{B[x_0,r_{t_{n-k+1}}]}) + 2^{-(n-k+1)} \leq \ldots \leq \nu(\hat{A}_{x_n}) + 2^{-(n-k)}
  \end{align*}
  for all $n \in \N$ and $k \in \set{1, \ldots, n}$.
  
  Furthermore, by repeatedly applying the second part of Lemma \ref{lem:lower_norm_localization}, for $n\in\N$ and $l\in\set{1,\ldots,n}$ we obtain
  \begin{align*}
    \nu(\hat{A}_{y_{n}^{n-l}}|_{B[x_0,r_{t_l}]}) &\geq \nu(\hat{A}_{y_{n}^{n-l+1}}|_{B[x_0,r_{t_l}+r_{t_l} + r_{t_{l-1}}]}) = \nu(\hat{A}_{y_{n}^{n-l+1}}|_{B[x_0,2r_{t_l} + r_{t_{l-1}}]})\\
    &\geq \nu(\hat{A}_{y_{n}^{n-l+2}}|_{B[x_0,2r_{t_l}+r_{t_{l-1}} + r_{t_{l-1}} + r_{t_{l-2}}]}) = \nu(\hat{A}_{y_{n}^{n-l+2}}|_{B[x_0,2r_{t_l}+2r_{t_{l-1}} + r_{t_{l-2}}]}) \\
    &\geq \ldots \geq \nu(\hat{A}_{y_{n}^{n}}|_{B[x_0,2r_{t_l}+2r_{t_{l-1}} + \ldots + 2r_{t_{1}} + r_{t_{0}}]}) \geq \nu(\hat{A}_{y_{n}^{n}}|_{B[x_0,4r_{t_l}]}),
  \end{align*}
  where we have used that $r_{t_{k+1}} > 2r_{t_k}$ for all $k\in\N$.
  
  Now, let $y_n:=y_n^n$ for all $n\in\N$. As $\Gamma X$ is compact, $(y_n)_{n \in \N}$ has a convergent subnet denoted by $(y_{n_{\iota}})$. Let us denote the limit of this subnet by $y$. Corollary \ref{cor:hats} implies $\norm{(\hat{A}_{y_{n_\iota}} - \hat{A}_y) M_{\1_{B[x_0,4r_{t_l}]}}} \to 0$ as $y_{n_{\iota}} \to y$. By Lemma \ref{lem:estimate_lower_norm}, this also implies $\nu(\hat{A}_{y_{n_\iota}}|_{B[x_0,4r_{t_l}]}) \to \nu(\hat{A}_{y}|_{B[x_0,4r_{t_l}]})$. Thus,
  \begin{align*}
    \nu(\hat{A}_y) & \leq \nu(\hat{A}_{y}|_{B[x_0,4r_{t_l}]}) = \lim_{\iota} \nu(\hat{A}_{y_{n_\iota}}|_{B[x_0,4r_{t_l}]}) \leq \lim_{\iota} \nu(\hat{A}_{y_{n_\iota}^{n_\iota-l}}|_{B[x_0,r_{t_l}]})\\
    & \leq \lim_{\iota} \nu(\hat{A}_{x_{n_\iota}}) + 2^{-l} = \inf \set{\nu(\hat{A}_y):\; y\in \Gamma X} + 2^{-l}.
  \end{align*}
  Taking $l\to\infty$, we obtain $\nu(\hat{A}_y) = \inf \set{\nu(\hat{A}_y):\; y\in \Gamma X}$.  
\end{proof}

\begin{theorem}
\label{thm:Fredholm_equivalence}
  Let $A\in\calA^p$. The following are equivalent:
  \begin{enumerate}
    \item
      $A$ is Fredholm.
    \item
      For all $x\in\Gamma X$: $A_x$ is invertible, and $\sup_{x\in\Gamma X} \norm{A_x^{-1}}<\infty$.
    \item
      For all $x\in\Gamma X$: $A_x$ is invertible.
    \item
      For all $x\in\Gamma X$: $\hat{A}_x$ is invertible.
  \end{enumerate}
  Moreover, if any of the above is true, then $\norm{A_x^{-1}}\leq \norm{(A+\calK(M^p))^{-1}}$.
\end{theorem}

\begin{proof}
  We will prove the theorem in the following order: (a) $\Rightarrow$ (b) $\Rightarrow$ (c) $\Rightarrow$ (d) $\Rightarrow$ (b) $\Rightarrow$ (a). As everything is already covered above, only a few notes are necessary.
  
  $\bullet$ ``(a) $\Rightarrow$ (b)'': This is Theorem \ref{thm:Fredholm_implies_lim-ops_invertible}.
  
  $\bullet$ ``(b) $\Rightarrow$ (c)'': This is clear.
  
  $\bullet$ ``(c) $\Rightarrow$ (d)'': If $x\in\Gamma X$ and $A_x$ is invertible, then $(\hat{A}_x)^{-1} = A_x^{-1} P_x + Q_x$ and thus $\hat{A}_x$ is invertible.
  
  $\bullet$ ``(d) $\Rightarrow$ (b)'': Proposition \ref{prop:min_of_nu} implies $\sup\limits_{x\in\bdry{X}} \|\hat{A}_x^{-1}\|<\infty$. As $A_x^{-1} = \hat{A}_x^{-1}|_{\ran P_x}$, we get $\sup\limits_{x\in\bdry{X}} \norm{A_x^{-1}}<\infty$.
  
  $\bullet$ ``(b) $\Rightarrow$ (a)'': This is Theorem \ref{thm:invertiblility_and_uniform_bound_implies_Fredholm}.
  
  $\bullet$ The last statement is again Theorem \ref{thm:Fredholm_implies_lim-ops_invertible}.
\end{proof}

By exploiting the equivalence of (a) and (c) in terms of spectral theory, we obtain the following corollary.

\begin{corollary}
\label{cor:spectral_equality}
  Let $A\in\calA^p$. Then
  \[\bigcup_{x\in \bdry{X}} \sigma(A_x) = \sigma_{\ess}(A).\]
\end{corollary}

\section{Extension to different compactifications}
\label{sec:compactifications}

In this short section we extend Theorem \ref{thm:Fredholm_equivalence} and Corollary \ref{cor:spectral_equality} to other compactifications of $X$. This can be achieved rather easily due to the fact that the Stone-\u{C}ech-compactification has a universal property. 

\begin{proposition}
\label{prop:different_compactification}
Let $\tilde{X}$ be a compactification of $X$ and $\partial X := \tilde{X} \setminus X$ the boundary. Assume $A\in\calA^p$ and that for all $y \in \partial X$ and all nets $(y_\iota)_\iota$ in $X$ with $y_\iota\to y$ the following limits exist
  \[P_y := \wlim\limits_{y_\iota\to y} U_{y_\iota}^p P(U_{y_\iota}^p)^{-1}, \quad A_y := \wlim\limits_{y_\iota\to y} U_{y_\iota}^p AP(U_{y_\iota}^p)^{-1}|_{\ran(P_y)}.\]
Then $A$ is Fredholm if and only if $A_y$ is invertible for all $y \in \partial X$. Moreover,
\[\bigcup\limits_{y \in \partial X} \sigma(A_y) = \sigma_{\ess}(A).\]
\end{proposition}

\begin{proof}
Let $x \in \Gamma X$ and choose a net $(x_\iota)_\iota$ in $X$ with $x_\iota\to x$. By compactness, there is a subnet of $(x_\iota)_\iota$ that converges in $\tilde{X}$ to some $y \in \partial X$. As weak limits are unique, we have $P_y = P_x$ and $A_y = A_x$. Hence every limit operator with respect to $\beta X$ is also a limit operator with respect to the compactification $\tilde{X}$. Conversely, every limit operator with respect to $\tilde{X}$ is also a limit operator with respect to $\beta X$. Hence, since the sets of limit operators are the same, the proposition follows from Theorem \ref{thm:Fredholm_equivalence}.
\end{proof}

\begin{remark} \label{rem:sequences}
In fact, it is even possible to replace the nets by sequences. Let $A\in\calA^p$. Our assumptions on the measure space imply that $L_p(X,\mu)$ is separable and hence the strong operator topology on $\calL\bigl(L_p(X,\mu)\bigr)$ is metrizable on bounded sets. Therefore the sequential closure of the bounded set $\set{U_y^pAP(U_y^p)^{-1} :\; y \in X}$ coincides with the topological closure. In particular, every limit operator of $A$ is the strong limit of a sequence $\bigl(U_{y_n}^pAP(U_{y_n}^p)^{-1}\bigr)_{n \in \N}$ with $y_n \in X$. However, not every sequence $(y_n)_{n \in \N}$ tending to infinity (i.e.~$d(x_0,y_n) \to \infty$) yields a limit operator as the strong limit may not exist. As the closure of $\set{U_y^pAP(U_y^p)^{-1} :\; y \in X}$ is strongly compact by Theorem \ref{thm:limit_operator_existence}, every such sequence has a strongly convergent subsequence, though.
\end{remark}

\section{Applications and special cases}
\label{sec:applications}

In this section we showcase the strength of Theorem \ref{thm:Fredholm_equivalence} by providing some examples which satify our Assumptions \ref{ass:space}-\ref{ass:shifts}. Some of these examples are new and others are well-known. In the latter case we therefore rediscover known results.

\subsection{The trivial case}

Let $(X,d,\mu)$ be a metric measure space that satisfies Assumptions \ref{ass:space}-\ref{ass:shifts} and assume that $P$ is compact. Proposition \ref{prop:limit_operators_of_compact_operators} then implies that $P_x = 0$ for all $x \in \Gamma X$. Hence every possible limit operator acts only on the null space and is therefore invertible. This implies that every operator $A \in \calA^p$ is Fredholm. No surprise there, obviously.

\subsection{Operators on $\Z^n$}

Here, we recover the classical setup for limit operators. We have $X:=\Z^n$ with the usual metric $d$, and $\mu$ the counting measure. Clearly, $(\Z^n,d)$ is proper, has bounded geometry and finite asymptotic dimension (hence property \Aprime{} by Proposition \ref{prop:asymp_dim}, for an explicit construction see \cite[p.~83]{Lindner2006}). Also, $\mu$ is a Radon measure, hence Assumption \ref{ass:space} is satisfied. Let $M^p:=L_p(\Z^n,\mu) = \ell_p(\Z^n)$. Then $P = I$ and $M_{\1_K}$ is compact for every compact subset $K\subseteq \Z^n$. This implies Assumption \ref{ass:subspaces_and_projection}. Let $x_0 = 0$ and $\phi_x(y) = y+x$ for all $y \in \Z^n$. Since $\mu$ is translation invariant, $h_x=\1$, and the operators $U_x^p$ are just the usual shift operators. Therefore Assumption \ref{ass:shifts} is satisfied as well. Thus, Theorem \ref{thm:Fredholm_equivalence} recovers \cite[Theorem 11]{LindnerSeidel2014} in the scalar valued case (cf.~Remark \ref{rem:sequences}).

In fact, any Radon measure $\mu$ with $\mu(\set{x}) \neq 0$ for all $x \in \Z^n$ will work as all the other assumptions are trivial in this case.

\subsection{Operators on $\N^n$}

Choosing $M^p = \ell_p(\N^n)$ with the canonical projection $P \from \ell_p(\Z^n) \to \ell_p(\N^n)$ yields another example. In this case the domain of a limit operator depends on $x \in \beta X$. In the case $n = 1$, for example, we have $P_x = I_{\ell_p(\Z)}$ if $x$ is a limit point at $+\infty$ and $P_x = 0$ if $x$ is a limit point at $-\infty$. For $n > 1$ the situation is much more complicated as every direction in $\Z^n$ yields a different projection $P_x$ and hence a different domain for the limit operators.

\subsection{Discrete groups}

Let $(X,d)$ be an unbounded proper discrete group that satisfies property \Aprime{} (e.g.~by having finite asymptotic dimension or satisfying property \A{}, cf.~Theorem \ref{thm:A_equiv_A'} and Proposition \ref{prop:asymp_dim}). Let $\mu$ be the counting measure and $\phi_x \from X \to X$ the group action $\phi_x(y) = y \circ x$ with $x_0 = e$, the identity element. Choose any subset (not necessarily a subgroup) $Y \subseteq X$ and set $M^p = \ell_p(Y) := L_p(Y,\mu)$. $M^p$ is a closed subspace of $\ell_p(X) = L_p(X,\mu)$ and the canonical projection $P$ satisfies all assumptions. To see that $x \mapsto M_{\1_K}U_x^pP(U_x^p)^{-1}M_{\1_{K'}}$ extends continuously to $\beta X$ for all compact subsets $K,K' \subseteq X$ just observe that this is equivalent to extending $x \mapsto U_x^pP(U_x^p)^{-1}$ with respect to the weak operator topology, which is always possible by Proposition \ref{prop:weak_limit_operators}.

\subsection{The Fock space ($\C^n$)}

Let $X = \C^n$ with the Euclidean metric $d$ and $p \in (1,\infty)$. $(X,d)$ is proper, has bounded geometry and finite asymptotic dimension. Let $\mu$ be a Gaussian measure on $\C^n$, i.e.~$\mathrm{d}\mu(z) = \left(\frac{p\alpha}{2\pi}\right)^n e^{-\frac{p\alpha}{2}|z|^2} \mathrm{d}z$ for some $\alpha > 0$, and $M^p$ the closed subspace of entire functions in $L_p(\C^n,\mu)$. By \cite[Theorem 7.1]{JansonPeetreRochberg1987} the orthogonal projection $P \from L_2(\C^n,\mu) \to M^2$ extends to a bounded projection $P \from L_p(\C^n,\mu) \to M^p$. For $\phi_x$ we choose the usual translations given by $\phi_x(y) = y+x$ for all $x,y \in \C^n$. Then
\[\frac{\mathrm{d}(\mu \circ \phi_x)}{\mathrm{d}\mu}(y) = e^{-\frac{p\alpha}{2}|x|^2-p\alpha\Re\dupa{y}{x}}.\]
It is natural to choose $h_x(y) := e^{-\frac{\alpha}{2}|x|^2-\alpha\dupa{y}{x}}$ so that $h_x$ is an entire function again. As $P$ is given explicitly as an integral operator, it is easy to check that $M_{\1_K}P$, $PM_{\1_K} \in \calK(L_p(\C^n,\mu))$ for all $K \subseteq X$ compact and $U_x^pP = PU_x^p$ (see e.g.~\cite[Proposition 7, Lemma 17]{FulscheHagger2019}). Therefore the assumptions are again satisfied and we recover \cite[Theorem 28]{FulscheHagger2019}.

\subsection{Generalized Fock spaces}

Choosing $X = \C^n$ as above but changing the measure $\mu$ gives rise to more interesting examples. For example, let us choose $n = 1$, $m \in \N$, $p \in (1,\infty)$ and $\mathrm{d}\mu(z) = \left(\frac{p}{2}\right)^{\frac{mp}{2}+1}\frac{1}{\pi\Gamma(\frac{mp}{2}+1)}\abs{z}^{mp}e^{-\frac{p}{2}\abs{z}^2} \, \mathrm{d}z$, where the constant is chosen such that $\mu(\C) = 1$. Let $M^p$ be the closed subspace of entire functions in $L_p(\C,\mu)$ (the so-called Fock-Sobolev spaces \cite{ChoZhu2012}) and $P \from L_2(\C,\mu) \to M^2$ the orthogonal projection. $P$ is given as the integral operator
\[[Pf](z) = \frac{1}{\pi m!}\int_{\C} f(w)K(z,w)\abs{w}^{2m}e^{-\abs{w}^2} \, \mathrm{d}w,\]
where $K(z,w) = \frac{m!}{(z\bar{w})^m}\Bigl(e^{z\bar{w}}-\sum\limits_{k = 0}^{m-1} \frac{(z\bar{w})^k}{k!}\Bigr)$ is the reproducing kernel. As for the standard Fock spaces, $P$ extends to a bounded projection $P \from L_p(\C,\mu) \to M^p$. With the usual dual pairing, we have
\begin{align*}
\dupa{Pf}{g} &= \left(\frac{1}{\pi m!}\right)^2\int_{\C} \int_{\C} f(w)K(z,w)\abs{w}^{2m}e^{-\abs{w}^2} \, \mathrm{d}w \, g(z)\abs{z}^{2m}e^{-\abs{z}^2} \, \mathrm{d}z\\
&= \left(\frac{1}{\pi m!}\right)^2\int_{\C} \int_{\C} f(w)\abs{w}^me^{-\frac{1}{2}\abs{w}^2} K(z,w)\abs{zw}^me^{-\frac{1}{2}(\abs{z}^2+\abs{w}^2)} \, \mathrm{d}w \, g(z)\abs{z}^{m}e^{-\frac{1}{2}\abs{z}^2} \mathrm{d}z
\end{align*}
for $f \in L_p(\C,\mu)$, $g \in L_q(\C,\mu)$ and $\frac{1}{p} + \frac{1}{q} = 1$. The estimate
\[\abs{K(z,w)}\abs{zw}^me^{-\frac{1}{2}(\abs{z}^2+\abs{w}^2)} \leq Ce^{-\delta|z-w|^2}\]
for constants $C,\delta > 0$ now immediately implies $P \in \BDO^p$ because it is approximated by the integral operators with kernel $K(z,w)\1_{B(0,R)}(z-w)$ as $R \to \infty$. Similarly, using the Hille-Tamarkin theorem (a variant of Hilbert-Schmidt for $p \neq 2$, see e.g.~\cite[Theorem 41.6]{Zaanen}) one can easily show $M_{\1_K}P$, $PM_{\1_K} \in \calK(L_p(\C,\mu))$. Finally, we show that $x \mapsto M_{\1_K}U_x^pP(U_x^p)^{-1}M_{\1_{K'}}$ has a continuous extension to $\beta\C$ for all compact subsets $K,K' \subseteq \C$. We have
\[\frac{\mathrm{d}(\mu \circ \phi_x)}{\mathrm{d}\mu}(y) = \frac{|y+x|^{mp}}{|y|^{mp}} e^{-\frac{p}{2}|x|^2-p\Re\dupa{y}{x}},\]
therefore we choose $h_x(y) := \frac{(y+x)^m}{y^m}e^{-\frac{1}{2}|x|^2-\dupa{y}{x}}$. It follows
\begin{align*}
&[U_x^pP(U_x^p)^{-1}f](z)\\
&= \frac{1}{\pi m!}\int_{\C} f(y)\frac{y^m}{(y+x)^m}e^{\frac{1}{2}|x|^2+y\bar{x}}\frac{m!}{(z+x)^m(\bar{y}+\bar{x})^m}\left(e^{(z+x)(\bar{y}+\bar{x})} - \sum\limits_{k = 0}^{m-1} \frac{(z+x)^k(\bar{y}+\bar{x})^k}{k!}\right)\\
&\qquad \qquad \qquad |y+x|^{2m}e^{-|y+x|^2}\frac{(z+x)^m}{z^m}e^{-\frac{1}{2}|x|^2-z\bar{x}} \, \mathrm{d}y\\
&= \frac{e^{\frac{1}{2}|z|^2}}{\pi z^m}\int_{\C} f(y)\Bigl(e^{-\frac{1}{2}|z-y|^2}e^{i\Im(z\bar{y})} - \sum\limits_{k = 0}^{m-1} \frac{(z+x)^k(\bar{y}+\bar{x})^k}{k!}e^{-\frac{1}{2}|y+x|^2-\frac{1}{2}|z+x|^2}e^{i\Im(y\bar{x})+i\Im(x\bar{z})}\Bigr)\\
&\qquad \qquad \qquad \quad y^m e^{-\frac{1}{2}|y|^2} \, \mathrm{d}y.
\end{align*}
Define
\[[P_{\infty}f](z) = \frac{e^{\frac{1}{2}|z|^2}}{\pi z^m}\int_{\C} f(y)e^{-\frac{1}{2}|z-y|^2}e^{i\Im(z\bar{y})} y^m e^{-\frac{1}{2}|y|^2} \, \mathrm{d}y.\]
Let $\frac{1}{p} + \frac{1}{q} = 1$. Then, using H\"older's inequality, we get
\begin{align*}
&\sup\limits_{\norm{f}_p = 1} \norm{M_{\1_K}\left(U_x^pP(U_x^p)^{-1} - P_{\infty}\right)M_{\1_{K'}}f}_p^p\\
&\qquad \qquad \qquad \qquad \leq c_{m,p} \sup\limits_{\norm{f}_p = 1}\int_K \left(\int_{K'} |f(y)|\sum\limits_{k = 0}^{m-1} \frac{|z+x|^k|\bar{y}+\bar{x}|^k}{k!}e^{-\frac{1}{2}|y+x|^2-\frac{1}{2}|z+x|^2}|y|^m e^{-\frac{1}{2}|y|^2} \, \mathrm{d}y\right)^p \, \mathrm{d}z\\
&\qquad \qquad \qquad \qquad \leq c_{m,p}' \int_K \left(\int_{K'} \left(\sum\limits_{k = 0}^{m-1} \frac{|z+x|^k|\bar{y}+\bar{x}|^k}{k!}e^{-\frac{1}{2}|y+x|^2-\frac{1}{2}|z+x|^2}\right)^q \, \mathrm{d}y\right)^{\frac{p}{q}} \, \mathrm{d}z\\
&\qquad \qquad \qquad \qquad \to 0
\end{align*}
as $|x| \to \infty$, where $c_{m,p}$ and $c_{m,p}'$ are some irrelevant constants coming from the measure. Therefore we can extend $x \mapsto M_{\1_K}U_x^pP(U_x^p)^{-1}M_{\1_{K'}}$ continuously to $\Gamma\C$ by $P_x = P_{\infty}$ for all $x \in \Gamma\C$.

\subsection{Operators on Bergman spaces}

Let $X$ be a bounded symmetric domain in its standard realization as a convex circular subset of $\C^n$ containing the origin and let $d$ be the corresponding Bergman metric. It is known that $(X,d)$ is proper, unbounded (as a metric space), has bounded geometry and finite asymptotic dimension.

Let $h$ denote the Jordan triple determinant\footnote{We refer to \cite{Englis1999} and \cite{FarautKoranyi1990} for definitions and properties of $h$.} corresponding to $X$ and let $\mu$ be the standard weighted Lebesgue measure given by
\[\mathrm{d}\mu(z) = c_{\nu} h(z,z)^{\nu} \, \mathrm{d}z,\]
where $\nu > -1$ and $c_{\nu}$ is chosen such that $\mu(X) = 1$. It is clear that $\mu$ is a Radon measure. Let $M^p$ denote the subspace of holomorphic functions in $L_p(X,\mu)$. It is shown in \cite{Hagger2017} that for suitable values of $p$ and $\nu$ there is a (bounded) band-dominated projection $P \from L_p(X,\mu) \to M^p$. Moreover, it is shown that in these cases $M_{\1_K}P$ and $PM_{\1_K}$ are compact for compact subsets $K \subseteq X$.

For every $x \in X$ we can choose biholomorphic maps $\phi_x \from X \to X$ with $\phi_x(0) = x$ and $x \mapsto \phi_x(y)$ continuous for all $y \in X$ (\cite[Lemma 5]{Hagger2019}). Moreover, the standard transformation properties of $\mu$ show that $\mu\circ\phi_x \ll \mu \ll \mu\circ \phi_x$ and we can choose
\[h_x(y) = \frac{h(x,x)^{\frac{\nu+g}{p}}}{h(y,x)^{\frac{2(\nu+g)}{p}}},\]
where $g$ denotes the genus of $X$ (see again \cite{Englis1999,Hagger2019}). As $h(y,x)$ is actually a polynomial in $y$ and $\overline{x}$, and has no zeroes in $X \times X$, $x \mapsto h_x(y)$ is continuous for all $y \in X$ as well. It remains to show that $x \mapsto M_{\1_K}U_x^pP(U_x^p)^{-1}M_{\1_{K'}}$ extends continuously to the Stone-\v{C}ech compactification $\beta X$. In \cite{Hagger2019} a differently weighted projection was used for $p \neq 2$ to achieve $U_x^pP = PU_x^p$, which trivializes the condition. Here we stick to the standard projection to show what happens. Using the explicit formula for $P$, i.e.
\[[Pf](z) = \int_X f(y)h(z,y)^{-\nu-g} \, \mathrm{d}\mu(y),\]
and the standard transformation properties of $h$ and $\mu$ (see again \cite{Englis1999,Hagger2019}), we arrive at
\[[U_x^pP(U_x^p)^{-1}f](z) = h(z,x)^{(\nu+g)(1-\frac{2}{p})}\int_X f(y)h(y,x)^{-(\nu+g)(1-\frac{2}{p})}h(z,y)^{-\nu-g} \, \mathrm{d}\mu(y).\]
Choosing $g_x(z) := h(z,x)^{(\nu+g)(1-\frac{2}{p})}$, we therefore get 
\[U_x^pP(U_x^p)^{-1} = M_{g_x}PM_{g_x^{-1}}.\]
As $h$ is a polynomial, it extends continuously to $K \times \overline{X}$, where $K \subseteq X$ is compact and $\overline{X}$ is the Euclidean closure of $X$. Moreover, as the extension has no zeroes in $X \times \overline{X}$, it is bounded from below on $K \times \overline{X}$. This implies that $x \mapsto M_{g_x\1_K}$ and $x \mapsto M_{g_x^{-1}\1_{K'}}$ extend continuously to the Euclidean closure of $X$ and hence to $\beta X$. This implies that $M_{\1_K}U_x^pP(U_x^p)^{-1}M_{\1_{K'}}$ extends continuously as well. Hence all assumptions are satisfied and we recover \cite[Theorem B]{Hagger2019}.

The above obviously also works for reducible domains like the polydisk. More generally, if two spaces $(X,d_X,\mu_X)$ and $(Y,d_Y,\mu_Y)$ both satisfy the Assumptions \ref{ass:space} - \ref{ass:shifts}, then the product space $X \times Y$ also satisfies the assumptions with the obvious choices for the measure etc.

\subsection{Pluriharmonic Bergman spaces}

Consider the Bergman spaces from above and replace the space of holomorphic functions by the space of pluriharmonic functions, i.e.
\[M^p := \set{f \from X \to \C : f = f_1 + \bar{f}_2; \, f_1,f_2 \text{ holomorphic}, f_2(0) = 0}.\]
The corresponding projection is $P + \overline{P} - P_{\C}$, where
\[[\overline{P}f](z) = \int_X f(y) h(y,z)^{-\nu-g} \, \mathrm{d}\mu(y)\]
and $P_{\C}$ is the rank-$1$ projection onto the constant functions. As $\overline{P}$ is very similar to $P$, $\overline{P} \in \BDO^p$ and the compactness of $M_{\1_K}\overline{P}$ and $\overline{P}M_{\1_K}$ are clear (just repeat the proofs for $P$ in \cite{FulscheHagger2019}). Moreover, we have
\[[U_x^p\overline{P}(U_x^p)^{-1}f](z) = \frac{h(x,z)^{\nu+g}}{h(z,x)^{\frac{2(\nu+g)}{p}}}\int_X f(y)\frac{h(y,x)^{\frac{2(\nu+g)}{p}}}{h(x,y)^{\nu+g}}h(y,z)^{-\nu-g} \, \mathrm{d}\mu(y).\]
As above this implies that $x \mapsto M_{\1_K}U_x^p\overline{P}(U_x^p)^{-1}M_{\1_{K'}}$ extends continuously to $\beta X$. As $P_{\C}$ is compact, the corresponding assumptions are trivially satisfied. It follows that $P + \overline{P} - P_{\C}$ also satisfies the assumptions. Theorem \ref{thm:Fredholm_equivalence} therefore also applies to pluriharmonic Bergman spaces.

More generally, if we have two projections $P,P'$ that satisfy the assumptions, then $P+P'$ also satisfies the assumptions provided that $P+P'$ is again a projection, i.e.~if $PP' = P'P = 0$. Similarly, using Lemma \ref{lem:compact_convergence}, $PP'$ satisfies the assumptions if $P$ and $P'$ commute. Hence Theorem \ref{thm:Fredholm_equivalence} also holds with $M^p := \ran(P+P')$ and $M^p := \ran(PP')$, respectively.

\subsection{Vector valued function spaces}

Assume that $(X,d,\mu)$ satisfies Assumptions \ref{ass:space} - \ref{ass:shifts}. Since the vector valued $L_p$-space
\[L_p(X,\mu;\C^n) := \set{f \from X \to \C^n : \norm{f}_p^p := \sum\limits_{k = 1}^n \norm{f_k}_p^p < \infty}\]
is isometrically isomorphic to $L_p(X \times \set{0,\ldots,n-1},\mu \times \nu)$, where $\nu$ is the counting measure on $\set{0,\ldots,n-1}$, we can also cover these cases here. For the metric we choose $d'(x',y') := d(x,y) + \min\{|j-k|,|j-k+n|,|j-k-n|\}$ for $x' := (x,j)$ and $y' := (y,k)$ in $X' := X \times \set{0,\ldots,n-1}$. Clearly, $(X',d')$ is again a proper metric space of bounded geometry that satisfies property \Aprime{}. Moreover, $\mu' := \mu \times \nu$ is obviously Radon. As our closed subspace we choose $M^p \times \set{0,\ldots,n-1}$ and our projection is $P \otimes I_{\set{0,\ldots,n-1}}$. Moreover, for $x' = (x,j)$ and $y' = (y,k)$ we set $\phi_{x'}(y') := (\phi_x(y),k+j \mod n)$. Obviously, the neutral element is $(x_0,0)$ here. All assumptions now follow directly from the corresponding assumptions for $X$.

\begin{remark}
  In order to cover also Banach space-valued $L_p$-spaces $L_p(X,\mu;Y)$ with an infinite-dimensional Banach space $Y$ (as well as the cases $p\in\set{1,\infty}$) we can employ the so-called $\mathcal{P}$-theory; cf.\ \cite{RabinovichRochSilbermann2004,Seidel2014}. This extension may appear elsewhere.
\end{remark}

\noindent
Raffael Hagger \\
Leibniz-Universit\"at Hannover \\
Fakult\"at f\"ur Mathematik und Physik \\
Fachgebiet Mathematik \\
Welfengarten 1\\
30167 Hannover, Germany \\
{\tt raffael.hagger@math.uni-hannover.de}\\
and\\
University of Reading \\
Department of Mathematics and Statistics\\
Whiteknights, PO Box 220\\
Reading RG6 6AX, United Kingdom\\
{\tt r.t.hagger@reading.ac.uk}

\bigskip

\noindent
Christian Seifert \\
Technische Universit\"at Hamburg \\
Institut f\"ur Mathematik \\
Am Schwarzenberg-Campus 3 \\
21073 Hamburg, Germany \\
{\tt christian.seifert@tuhh.de}\\
and\\
Technische Universit\"at Clausthal \\
Insitut f\"ur Mathematik \\
Erzstra{\ss}e 1\\
38678 Clausthal-Zellerfeld, Germany \\
{\tt christian.seifert@tu-clausthal.de}

\end{document}